\documentclass[11pt]{article}
\usepackage{mathrsfs}
\usepackage{amsmath,amssymb}
\usepackage{amsthm,bm}
\usepackage[colorlinks,linkcolor=blue,anchorcolor=blue,citecolor=green]{hyperref}
\usepackage{graphics,graphicx}
\usepackage{color}
\usepackage{enumerate}

\usepackage[title]{appendix}
\usepackage{subfigure}

\usepackage{caption}

\allowdisplaybreaks

\usepackage{tikz}
\usepackage{indentfirst}

\newtheorem{Thm}{Theorem}[section]
\newtheorem{theorem}[Thm]{Theorem}

\newtheorem{definition}[Thm]{Definition}

\newtheorem{lemma}[Thm]{Lemma}

\newtheorem{corollary}[Thm]{Corollary}

\newtheorem{proposition}[Thm]{Proposition}

\newtheorem{remark}[Thm]{Remark}

\def\E{{\mathbb{E}}}
\def\P{{\mathbb{P}}}

\def\b{{\mathfrak{b}}}
\def\dd{{\mathfrak{d}}}
\def\pp{{\mathbf{p}}}
\def\qq{{\mathbf{q}}}

\title{Tail distributions of cover times of once-reinforced random walks
	\footnote{Y. Liu is supported by CNNSF (No. 11731009, No.11926327, No. 12231002) and Center for Statistical Science, PKU.
		K. Xiang is supported by CNNSF (No.~12171410) and by Hu Xiang Gao Ceng Ci Ren Cai Ju Jiao Gong Cheng-Chuang Xin Ren Cai (No. 2019RS1057).}}
\author{
	Xiangyu Huang\footnote{YMSC, Tsinghua University, Beijing 100084, China. \emph{xiangyuhuang077@gmail.com}}
	\and
	Yong Liu\footnote{LMAM, School of Mathematical Sciences, Peking University, Beijing 100871, China. \emph{liuyong@math.pku.edu.cn}}
	\and
	Kainan Xiang\footnote{School of Mathematics and Computational Science, Xiangtan University, Xiangtan City 411105, Hunan Province, China. \emph{kainan.xiang@xtu.edu.cn}}}
\date{}

\usepackage{caption}
\usepackage{amsthm,amsmath,amsfonts,amssymb,mathrsfs}
\usepackage{graphics,graphicx}
\usepackage{bm}

\usepackage{natbib}
\usepackage[title]{appendix}
\usepackage{subfigure}

\usepackage{tikz}
\usepackage{indentfirst}

\begin{document}

\maketitle
\begin{abstract}
	We consider the tail distribution of the edge cover time of a specific non-Markov process, 
	$\delta$ once-reinforced random walk, on finite connected graphs, whose transition probability is proportional to weights of edges. Here the weights are $1$ on edges not traversed and $\delta\in(0,\infty)$ otherwise.
	In detail, we show that its tail distribution decays exponentially, and obtain a phase transition of the exponential integrability of the edge cover time with critical exponent $\alpha_c^1(\delta)$, which has a variational representation and some interesting analytic properties including $\alpha_c^1(0+)$ reflecting the graph structures.\\

\end{abstract}

	
	\section{Introduction}\label{sec 1}
\setcounter{equation}{0}

	\noindent For all connected locally finite undirected graphs $G=(V,E)$\label{G=(V,E)} with vertex set $V$ and edge set $E$, we write $u\sim v$\label{u sim v} if $u,v\in V$ are adjacent, and denote the corresponding edge by $uv=\{u,v\}$. For all subgraphs $G'=(V',E')\subseteq G$, we set $\partial E'$ to be the set of all edges $e'\in E'$ such that there is some $e\in E\setminus E'$ adjacent to $e'$.
	For each discrete stochastic process moving on the graph $G$, we define the edge cover time $C_E$ to be the stopping time that the process visits all edges at least once. Precisely, for any discrete process $(X_t)_{t\ge 0}$, set
	$$C_E=\inf\{t:\ \forall e\in E,\ \exists\,s\le t,\ X_{s-1}X_{s}=e \}.$$\label{cover time}
	Similarly we define the vertex cover time $C_V$ to be the stopping time for process covering all vertices. In literature, the mean cover time $t_{\rm cov}:=\max_{v\in V}\mathbb{E}_v C_V$ are always considered, where $\mathbb{E}_v$ is the expectation conditioned on process starting at $v$. For the relevant background material, please refer to survey \cite{L1993} and books \cite{AF2002, LPW2009}. Cover time is also useful in the study of computer science, due to its applications to designing universal traversal sequences, testing graph connectivity shown in \cite{AKLLR1979}, and protocol testing shown in \cite{MP1994}.
	
	There are many results on cover times of simple random walks (SRWs) in literatures. One important result was given by \cite{DLP2012}. They gave a sharp estimate for the mean cover time $t_{\rm cov}$ by the relationship between the discrete Gaussian free field and the SRW. According to this relationship, \cite{Z2018} proved the exponential decay of the tail distribution of vertex cover time for SRWs. For SRWs on $d$-dimensional torus, for $d=2$, \cite{CGP2013} gave a large deviation principle (LDP) result of the vertex cover time, and recently for $d\ge3$, \cite{LSX2024} also proved an LDP result by the relevance to random interlacements. \cite{DK2025} also showed an exponentially decay of the vertex cover time of SRWs on every finite graph. For some other graphs, \cite{M2024} showed an up to constant estimate of mean cover time $t_{\rm cov}$ for SRWs on dynamical percolation on $d$-dimensional torus in sub-critical case, using some results in \cite{PSS2015}. \cite{R2022} showed a suprising limit result on mean cover time $t_{\rm cov}$ for branching random walk on $d$-regular trees for $d\ge3$. On $b$-ary trees, \cite{A1991} analyzed the asymptotic behavior of $t_{\rm cov}$ of SRWs, and \cite{DRZ2021} observed the convergence in distribution of the vertex cover time of SRWs specifically for $b=2$.
	
 In this paper, we consider the once-reinforced random walk (ORRW) $X=(X_n)_{n\geq 0}$\label{X=(X_n)_{n>0}} on $G$, introduced by \cite{Dav1990}, which is a variation of the edge reinforced random walk (ERRW). The transition probability of ORRWs depends on the weight function defined on edges, where the weight on every edge at time $n$ is $1$ if this edge has not been traversed by the process, and $\delta$ otherwise. Here $\delta$ is a positive parameter called the reinforcement factor.
 The ORRW fits in a large family of self-interacting random walks, the study of which is very challenging due to the dependence on the whole past trajectory. Compared with the original model ERRW, the ORRW has a simplified weight function. However, the study of ORRW is suprisingly more challenging than other ERRWs. Up to now, even the most basic question, recurrence/transience, has no complete result on graphs with circles. For some relevant results, refer to \cite{Se2006, KS2018, CKS2017, HLSX2021, KSS2018}.

We aim to know more about the long time behaviors of ORRWs.
This paper is devoted to studying the critical exponent for exponential integrability of some stopping times of ORRWs, including the edge cover time of ORRWs on finite connected graphs. To the best of our knowledge, this paper is the first one considering the cover time of non-Markov process.
We always assume $\vert V\vert\geq 3$ in this paper due to the fact that $C_E\equiv 1$ for $|V|=2$. Set $\mathbb{P}_{x_0}$ to be the probability under the condition that the starting point $x_0$ is in $V$.
Our main results are summarized as follows:

\begin{theorem}[Critical exponent for exponential integrability of $C_E$]\label{Mthm C_E}
	$\mathbb{P}_{x_0}\left(C_E>n\right)$ decays exponentially with the rate $\alpha^1_c(\delta)\in (0,\infty)$, i.e.,
	\begin{equation}\label{exp-decay CE1}
		\lim_{n\to\infty}\frac{1}{n}\log \mathbb{P}_{x_0}\left(C_E>n\right)=-\alpha^1_c(\delta).
	\end{equation}
	More precisely,
	\begin{equation}\label{more alpha_c}
		0<\liminf_{n\to\infty}e^{n\alpha^1_c(\delta)}\mathbb{P}_{x_0} (C_E>n)\le \limsup_{n\to\infty}e^{n\alpha^1_c(\delta)}\mathbb{P}_{x_0} (C_E>n)<\infty.
	\end{equation}
	And $\alpha_c^1(\delta)$ is the critical exponent for exponential integrability of $C_E$ in the sense that
	$$
	\mathbb{E}_{x_0}\left[ e^{\alpha C_E} \right]<\infty \ \mbox{if}\ \alpha<\alpha^1_c(\delta),\ \mbox{and}\
	\mathbb{E}_{x_0}\left[ e^{\alpha C_E} \right]=\infty\ \mbox{if}\  \alpha\ge\alpha^1_c(\delta).
	$$
\end{theorem}

Theorem \ref{Mthm C_E} is a corollary of Theorem \ref{m-exp decay}. For more details of $\alpha_c^1$, see Theorem \ref{m-exp decay} and \eqref{eq-alpha-1}.

\begin{theorem}[Analytic property and asymptotic behaviour of $\alpha^1_c(\delta)$]
	\label{Mthm C_Edelta}
	$\mathbf {(a)}$  $\alpha_c^1(\delta)$ is continuous and strictly decreasing in $\delta>0$ and uniformly continuous in $\delta\ge\delta_0$ for any $\delta_0>0$.
	$\mathbf {(b)}$ $\lim\limits_{\delta\to \infty}\alpha_c^1(\delta)=0$. $\mathbf{(c)}$ $\lim\limits_{\delta\to 0+}\alpha_c^1(\delta)=\infty$ when $G$ is a 3-vertex connected graph or a star-shaped graph, and $\lim\limits_{\delta\to 0+}\alpha_c^1(\delta)<\infty$ otherwise.
\end{theorem}

The proof of Theorem \ref{Mthm C_Edelta} is stated in Section \ref{sec 4.1}.

	\begin{figure}[htbp]
	\subfigure[]{\begin{minipage}[t]{0.5\textwidth}
			\centering{\includegraphics[scale=0.5]{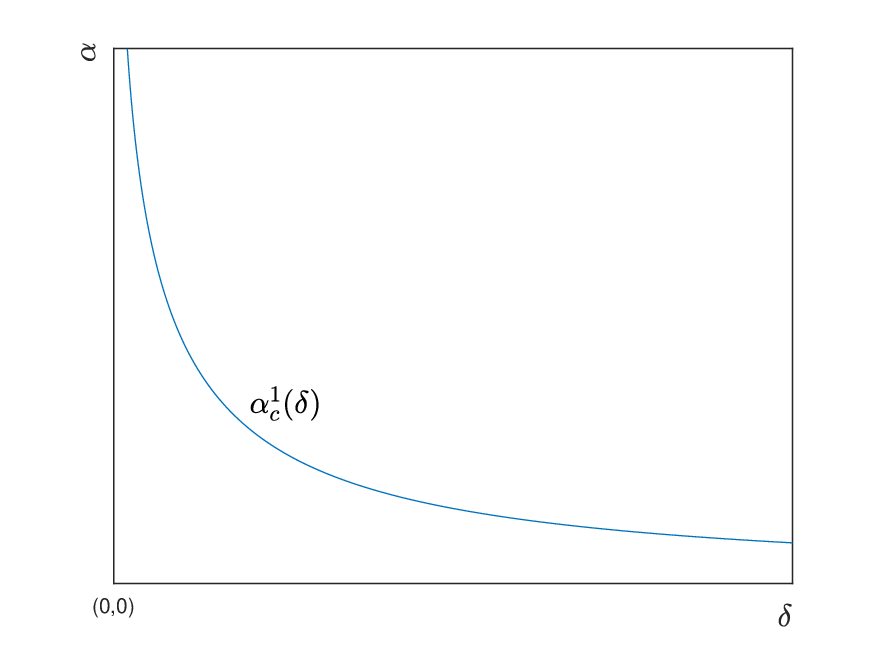}}
	\end{minipage}}
	\subfigure[]{\begin{minipage}[t]{0.5\textwidth}
			\centering{\includegraphics[scale=0.5]{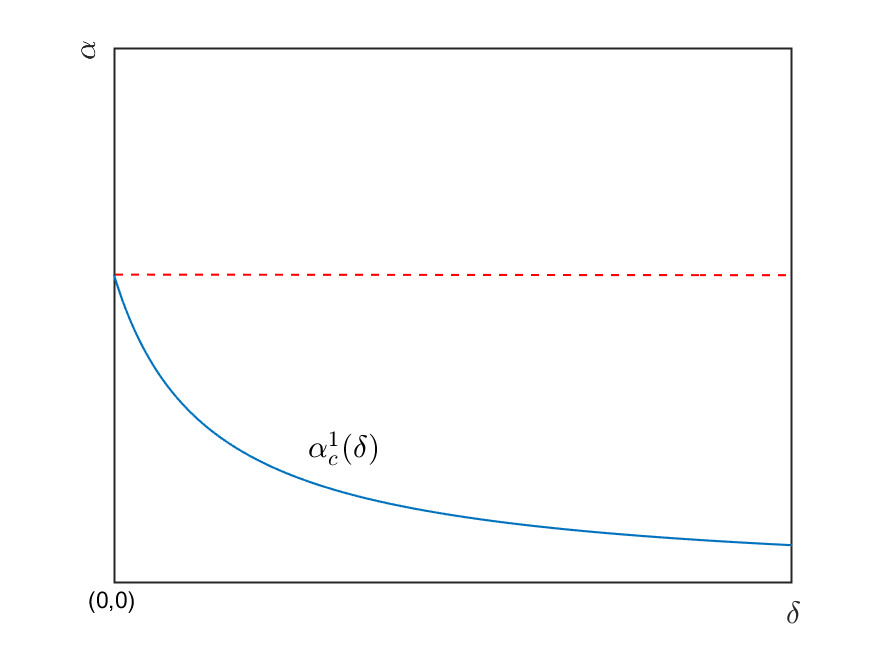}}
	\end{minipage}}
	\caption{\small{This is a sketch map of $\alpha_c^1(\cdot)$ for the ORRW on finite connected graphs. The left picture stands for $\alpha_c^1(\cdot)$ on  $3$-vertex connected graph or star-shaped graph. In this picture, $\alpha_c^1(\delta)$ converges to $\infty$ and $0$ as $\delta$ approaches $0$ and $\infty$ respectively.
			The right picture stands for $\alpha_c^1(\cdot)$ on other graphs.
			Here, $\alpha_c^1(\delta)$ converges to a positive number
			as $\delta$ approaches $0$, and converges to $0$ as $\delta$ approaches $\infty$.
	} }
	\label{phase line 0}
\end{figure}

Note that the edge cover time $C_E$ is a.s. finite on finite connected graphs, 
and that the $\delta$-ORRW becomes an SRW on $G$ after traversing all edges. These observations strongly suggest that long time behaviors of ORRWs should be the same as those of SRWs. However, on finite connected graphs, the LDP for empirical measures of ORRWs is different from that of the SRWs, see \cite{HLX2022}. Therefore, to understand the influence to the LDP, it is natural and interesting to study the edge cover time.

The model we consider, i.e., ORRWs, is a non-Markovian process. Due to the loss of Markov property, the methods in the literature on cover times of SRWs do not work well on the ORRWs. In this paper, we use the general LDP result of ORRWs shown in \cite{HLX2022} to study the tail distribution of the edge cover time $C_E$ (it also works for the vertex cover time). Additionally, we consider the critical exponent for exponential integrability of $C_E$. Moreover, we generalize these results to some more stopping times.

	In this paper, we introduce a lifted graph and use \emph{a general LDP for empirical measures of ORRWs} on this graph to give our proof. This general LDP considers a restricted Laplace functional in order to study the cover time and some other stopping times further. This method helps us to overcome the non-Markov property of ORRWs. A novel ingredient in our proofs is to provide an exponential decay (see Theorem \ref{m-exp decay}), which implies the critical exponent for exponential integrability of the cover time (see Theorem \ref{thm_prop_alpha_c}).

This paper is organized as follows. In Section \ref{sec 2}, we give core ideas and main results of the paper, introduce the weak convergence approach, and explain how it works in our study. In Section \ref{sec 4.1}, we prove the variational representation of critical exponent for for exponential integrability of the cover time. In Section \ref{sec 4.2}, we show the analytic property of this critical exponent. 
Finally, in Appendices, we give our proofs of some results such as propositions, lemmas and theorems appearing in the main text.

	\section{Main results and preliminaries}\label{sec 2}
\setcounter{equation}{0}
\noindent Our detailed main results are presented in Subsection \ref{sec 2.1}. In Subsection \ref{sec 2.15}, we lift the graph $G$ to a directed graph $S$, and state the related theorems on $S$ which are utilized to demonstrate the main results.

Let $X=(X_n)_{n\ge0}$ be the ORRW on $G$ with reinforcement factor $\delta$, and
$(\mathscr{F}_n)_{n\ge 0}$ be the natural filtration generated by the history of $X$.  We then give the definition of the ORRW $X$ by the following transition probability: for all $n\geq 0$, and $u\sim v$,

\[
\P\left(X_{n+1}=u\left|X_n=v,\mathscr{F}_n\right.\right)=\frac{w_n(uv)}{\mathop{\sum}_{u'\sim v}w_n(u'v)},
\]
where $w_n(e)$ is a random weight given by
\begin{equation}\label{weight}
	w_n(e)=1+(\delta-1)\cdot {\bm 1}_{\{N(e,n)>0 \}}=\left\{\begin{array}{ll}
		1 & {\rm if}\ N(e,n)=0,\\
		\delta & {\rm if}\ N(e,n)>0,
	\end{array}
	\right.
\end{equation}
where $\delta$ is a positive parameter called the reinforcement factor, $N(e,n):=|\{i<n:X_iX_{i+1}=e \}|$ is the number of times that $e$ has been traversed by the walk up to time $n$, and $|A|$ denotes the cardinality of any set $A$. If $\delta<1$, it is negatively reinforced; if $\delta>1$, it is positively reinforced. We write $\delta$-ORRW as the ORRW with parameter $\delta.$

\subsection{Main results}\label{sec 2.1}
\noindent Consider a $\delta$-ORRW $X=(X_n)_{n\geq 0}$ on a finite connected graph $G=(V,E)$ such that $X_0=x$ for some $x\in V$. Let $\b:=|E|\geq 2$ and $\mathbb{P}_{x_0}$ be the law of the $X=(X_n)_{n\geq 0}$ starting at $x_0$,  and $\mathbb{E}_{x_0}$ be the expectation under $\mathbb{P}_{x_0}$.

	\begin{definition}[Transition kernel on $G$]\label{Def for p_u on tree}
	\[
	\hat{p}_{E'}(x,y):=\frac{g(\delta,E',xy)}{\sum_{z\sim x}g(\delta,E',xz)},\ \text{for } y\sim x,
	\]
	where $E'\subseteq E$ and $g(\delta,E',xy):=\delta {\bf{1}}_{\{xy\in E' \}} + {\bf{1}}_{\{xy\notin E' \}}$. Specifically, $\hat{p}:=\hat{p}_E$ is the transition probability of the SRW on $G$.
\end{definition}

{

Let
\begin{eqnarray}\label{eq-edge-subset}
\mathscr{S}=\Bigg\{E':\ \exists\,\text{connected}\ G'=(V',E')\subseteq G=(V,E)\ \text{and}\ \exists\,v\in V\ \text{with}\
\{x,v\}\in E'\Bigg\}
\end{eqnarray}
be the set of edge subsets $E'$ of $E$ which induce a connected graph $G'=(V',E')$ with $V'$ being the set of all vertices of $E'$ and $x\in V'$ (noting $x$ is the starting point of the ORRW). Given any non-empty decreasing subset $\mathscr{S}_0$\label{mathscr{S}_0} of $\mathscr{S}$ with ``decreasing'' meaning that
$$\mbox{for any}\ E_1,E_2\in\mathscr{S}\ \mbox{with}\ E_1\subseteq E_2,\ \mbox{if}\ E_2\in\mathscr{S}_0,\ \mbox{then}\ E_1\in\mathscr{S}_0.$$
Consider a class of stopping times defined as
$$
\mathcal{T}_{\mathscr{S}_0}:=\inf\left\{n>0:\ \{e\in E: \exists m<n, X_mX_{m+1}=e \}\notin \mathscr{S}_0\right\}.
$$\label{mathcal{T}_{mathscr{S}_0}}
Notice edge cover time $C_E$ is a special case of the above stopping times:
$C_E=\mathcal{T}_{\mathscr{S}_1}\ \mbox{with}\ \mathscr{S}_1=\mathscr{S}\setminus\{E\}.$
Here,  $\mathscr{S}_1$ is also a decreasing subset of $\mathscr{S}$.

For any edge set $E'\subseteq E$, let $G'=(V',E')$ be the graph induced by $E'$ on all vertices of $E';$ and denote by $\mathscr{T}(E')$ the collection of the transition kernels $q$ from $V'$ to $V$ such that for any $x\in V'$ and $y\in V$, $q(x,y)>0$ only if $xy\in E'$. Particularly, $\mathscr{T}_G:=\mathscr{T}(E)$\label{mathscr{T}_G} is the set of the transition probabilities of Markov chains on the vertex set $V$ of the graph $G$, where the Markov chains move to their neighbours for each step.
For any Polish space $\mathscr{X}$, let $\mathscr{P}(\mathscr{X})$\label{mathscr{X}} be the set of all probability measures on $\mathscr{X}$ equipped with the weak convergence topology.

Now our main results on critical exponents for exponential integrability of stopping times $\mathcal{T}_{\mathscr{S}_0}$, particularly $C_E=\mathcal{T}_{\mathscr{S}_1}$, are given respectively by Theorem \ref{m-exp decay}, and Theorems \ref{Mthm C_E}-\ref{Mthm C_Edelta}.

\begin{theorem}[Critical exponent for exponential integrability of $\mathcal{T}_{\mathscr{S}_0}$]\label{m-exp decay}
Set
	\begin{equation}
	\alpha_c(\delta):=\inf_{\stackrel{\left(\nu,E_0,\hat{q}\right): E_0\in\mathscr{S}_0,\ \hat{q}\in\mathscr{T}(E_0)}{\nu\in\mathscr{P}(V),\ \nu \hat{q}=\nu,\ {\rm supp}(\nu)\subseteq V_0 }}\int_{V} R(\hat{q}\|\hat{p}_{E_0})\ {\rm d}\nu.\label{new exp. of alpha_c}
	\end{equation}	
\begin{itemize}
	\item[(a)]		
	$\mathbb{P}_{x_0}\left(\mathcal{T}_{\mathscr{S}_0}>n\right)$ decays exponentially with the rate $\alpha_c(\delta)\in (0,\infty)$, i.e.,
		\begin{equation}\label{exp-decay 1}
		\lim_{n\to\infty}\frac{1}{n}\log \mathbb{P}_{x_0}\left(\mathcal{T}_{\mathscr{S}_0}>n\right)=-\alpha_c(\delta).
		\end{equation}
	\item[(b)]	 More precisely,
	\begin{equation}\label{more alpha_c}
0<\liminf_{n\to\infty}e^{n\alpha_c(\delta)}\mathbb{P}_{x_0} (\mathcal{T}_{\mathscr{S}_0}>n)\le \limsup_{n\to\infty}e^{n\alpha_c(\delta)}\mathbb{P}_{x_0} (\mathcal{T}_{\mathscr{S}_0}>n)<\infty.
\end{equation}	
\item[(c)]  $\alpha_c(\delta)$ is the critical exponent for exponential integrability of $\mathcal{T}_{\mathscr{S}_0}$ in the sense that

\begin{equation*}\label{m-exp-crit}
	\mathbb{E}_{x_0}\left[ e^{\alpha \mathcal{T}_{\mathscr{S}_0}} \right]<\infty, \text{ if }  \alpha<\alpha_c(\delta), \text{ and }
	\mathbb{E}_{x_0}\left[ e^{\alpha \mathcal{T}_{\mathscr{S}_0}} \right]=\infty, \text{ if }  \alpha\ge\alpha_c(\delta).
\end{equation*}

\end{itemize}
		
	\end{theorem}

The proof of Theorem \ref{m-exp decay} is stated in Section \ref{sec 4.1}. Letting
\begin{equation}
	\alpha_c^1(\delta):=\inf_{\stackrel{\left(\nu,E_0,\hat{q}\right): E_0\in\mathscr{S}_1}{\nu \hat{q}=\nu,\, {\rm supp}(\nu)\subseteq E_0 }}\int_{V} R(\hat{q}\|\hat{p}_{E_0})\ {\rm d}\nu, \label{eq-alpha-1}
\end{equation}
and applying Theorem \ref{m-exp decay} to $C_E=\mathcal{T}_{\mathscr{S}_1}$ ($\mathscr{S}_1=\mathscr{S}\setminus \{E\}$), we have Theorem \ref{Mthm C_E}.  We still cannot provide some explicit expressions of the critical exponent $\alpha_c(\delta)$ and $\alpha^1_c(\delta)$. However,  through detailed analysis, we obtain some analytic properties and asymptotic behavior of  $\alpha^1_c(\delta)$ as $\delta\to 0$ and $\delta\to\infty$ in the following Theorem \ref{thm_prop_alpha_c}.

\begin{theorem}[Analytic property and asymptotic behaviour of $\alpha_c(\delta)$]\label{thm_prop_alpha_c}
	
	\begin{itemize}
		\item [(a)] $\alpha_c(\delta)$ is continuous and strictly decreasing in $\delta>0$ and uniformly continuous in $\delta\ge\delta_0$ for any $\delta_0>0$.
		
		\item[(b)]$\lim\limits_{\delta\to \infty}\alpha_c(\delta)=0$.
		\item [(c)] If for some edge $e\in E$, the collection of all edges adjacent to $e$, denoted by
		
		\noindent $E_e:=\{e'\in E:\ e'\cap e\neq\emptyset\},$
		is a subset of some $E_0\in\mathscr{S}_0$, then
		\[
		\lim_{\delta\to0}\alpha_c(\delta)=\inf_{\stackrel{(\nu,\E_0\in\mathscr{S}_0,\hat{q}):\ \nu\in\mathscr{P}(V), \nu\hat{q}=\nu}{ E({\rm supp}(\nu))\subseteq E_0\setminus \partial E_0 }}\int_V R(\hat{q}\|\hat{p})\ {\rm d}\nu<\infty,
		\]
		where $E(V')$ is the edge set of subgraph induced by the vertex set $V'$. Otherwise,		
		\[\lim_{\delta\to 0}\alpha_c(\delta)=\infty.\]
	\end{itemize}
\end{theorem}

 The proof of Theorem \ref{thm_prop_alpha_c} is stated in Section \ref{sec 4.2}. 
 Theorem \ref{Mthm C_Edelta} is derived from Theorem \ref{thm_prop_alpha_c}, see in Section \ref{sec 4.2}.
Moreover, due to Theorems \ref{m-exp decay} and \ref{thm_prop_alpha_c}, we can prove  a comparison theorem of tail probability of $\mathcal{T}_{\mathscr{S}_0}(\delta)$ with respect to different $\delta$.

 \begin{corollary}\label{stochastic inequality}
 	For any $0<\delta_1<\delta_2$, 
 	there exists some $N=N(\delta_1,\delta_2)$ such that
 	$$\mathbb{P}_{x_0}\left(\mathcal{T}_{\mathscr{S}_0}(\delta_1)>n\right)<\mathbb{P}_{x_0}\left(\mathcal{T}_{\mathscr{S}_0}(\delta_2)>n\right),\ n>N.$$
 \end{corollary}
 
 \begin{proof}[Proof of Corollary \ref{stochastic inequality}]
 	By Theorem \ref{m-exp decay}, for $i\in\{1,2\}$, we can verify that \[\lim_{n\to\infty}\frac{1}{n}\log\mathbb{P}_{x_0}\left(\mathcal{T}_{\mathscr{S}_0}(\delta_i)>n\right)=-\alpha_c(\delta_i).\]
 	By Theorem \ref{thm_prop_alpha_c},
 	$\alpha_c(\delta_1)>\alpha_c(\delta_2)>0$. Thus for any $\varepsilon\in\left(0,\alpha_c(\delta_2)\right)$, there is some $N=N(\varepsilon,\delta_1,\delta_2)$ such that for all $n>N$,
 	$
 	\exp\left(-n(\alpha_c(\delta_i)+\varepsilon)\right) <\mathbb{P}_{x_0}\left(\mathcal{T}_{\mathscr{S}_0}(\delta_i)>n\right)< \exp\left(-n(\alpha_c(\delta_i)-\varepsilon)\right),
 	$
 	$i=1,2.$
 	Then for small enough $\varepsilon$, $\mathbb{P}_{x_0}\left(\mathcal{T}_{\mathscr{S}_0}(\delta_1)>n\right)<\mathbb{P}_{x_0}\left(\mathcal{T}_{\mathscr{S}_0}(\delta_2)>n\right)$ for any $n>N$.\end{proof}

 We give Corollary \ref{stochastic inequality}, a comparison theorem of stopping times for different $\delta$, by some delicate exponential estimates in Theorem \ref{m-exp decay}. We are interested in some stronger comparison theorems to answer some questions on long time behaviors of the ORRWs. For instance, on infinite connected graph $G$, if $\delta_1<\delta_2$ and the $\delta_1$-ORRW is recurrent, is the $\delta_2$-ORRW recurrent?

 \begin{remark}
 	\rm	
 	The stopping times $\mathcal{T}_{\mathscr{S}_0}$ include not only $C_E$, but also some interesting stopping times such as hitting times and cover times of subgraphs.
 	
 	For any subset $E'\subset E$, set
 	\[
 	\mathscr{S}_{E'}^{\rm hitting}=\left\{E''\in \mathscr{S}:\ E''\subseteq E' \right\},\ \text{and}\ 
 	\mathscr{S}_{E'}^{\rm cover}=\left\{E''\in\mathscr{S}:\ E'\setminus E''\neq \emptyset\right\}.
 	\]
 	For the $\delta$-ORRW $(X_n)_{n\geq 0}$, $\mathcal{T}_{\mathscr{S}_{E'}^{\rm hitting}}=\inf\left\{n\geq 1:\ X_{n-1}X_n\notin E'\right\}$
 	is the hitting time of $E\setminus E'$, and $\mathcal{T}_{\mathscr{S}_{E'}^{\rm cover}}=\inf\left\{n\geq 1:\ E'\subseteq\{X_{m-1}X_m:\ m\le n\} \right\}$ is the edge cover time of subgraph induced by $E'$.
 	Thus we may obtain the similar properties to those in Theorem \ref{m-exp decay} and Corollary \ref{stochastic inequality} for both the hitting time and the edge cover time of subgraphs, where the discussion of edge cover times of subgraphs may be helpful to consider the recurrence. More precisely, for fixed $m$, if we can obtain a uniformly exponential integrability of cover times $C_{m,n}:=\inf\{ t: X_t\in [-n,n]^2 \text{ traverse all edges in } [-m,m]^2  \}$, which also implies its tightness, then we may deduce that the ORRW in $\mathbb{Z}^2$ traverses all edges in $[-m,m]^2$ in finite time $C_m$. This may yield a proof of the recurrence on $\mathbb{Z}^2$ by using Borel-Cantelli's lemma to $\mathbb{P}\big(X_{j+C_m} \text{ returns to }(0,0)\big)$. We will follow this idea to consider the recurrence of ORRWs on $\mathbb{Z}^2$.
 \end{remark}

\subsection{Lifted directed graph and generalized Laplace principle}\label{sec 2.15}

\noindent  In this subsection, we introduce the lifted directed graph. This graph was firstly introduced in \cite{HLX2022}, where the vertex set and edge set consist of all oriented edges of $G=(V,E)$ and all $2$-step oriented paths of $G$, respectively. We implement a stochastic process  $(\mathcal{Z}_n)_{n\ge 0}$ (see \eqref{Z-n}) on the lifted directed graph, the projection of which is the ORRW $(X_n)_{n\ge0}$ on the graph $G$. We consider its empirical process $(\mathcal{L}_n)_{n\ge 0}$.
This method helps us to turn the ORRW into some vertex-reinforced random walk on the lifted directed graph. Its transition probability can be determined if the current location of the random walk and the current empirical measure are known. Precisely, $(\mathcal{Z}_n, \mathcal{L}_n)_{n\ge 0}$ is a non-homogeneous Markov process.

Define a lifted directed graph $S$\label{S} associated with $G$ as follows:
\begin{itemize}
  \item Set the vertex set $V_S$ of $S$ to be $\{ \overrightarrow{uv}: u,v\in V, u\sim v \}$.
  \item  To character the directed edge set of $S$,  for  $\overrightarrow{uv}\in V_S$, set $z^+$ (resp. $z^-$), the head (resp. tail) of $z$, by that
	\begin{equation}
	z^+=v,\ z^-=u.\label{head_tail}
	\end{equation}

For every two vertices $z_1,z_2\in S$, denoted by $z_1\to z_2$, if $z_1^+=z_2^-$\label{z_1 to z_2}, and then the  directed edge  from $z_1$ to $z_2$ is defined as $\overrightarrow{z_1z_2}$ if $z_1\to z_2$.
\item For vertex set $V'\subseteq V_S$, let $\partial V'$ be all vertices $z'\in V'$ such that there is some $z\in V_S\setminus V'$ with $z'\to z$.  $\partial V'$ is the boundary vertex set of $V'$.
\end{itemize}

Set
\begin{equation}\label{Z-n}
 \mathcal{Z}_n:=\overrightarrow{X_nX_{n+1}},  \ n\ge 0.
 \end{equation}
Here, for any $n\geq 0$, $\mathcal{Z}_n^-=X_{n}$ and $\mathcal{Z}_n^+=X_{n+1}$. 

\begin{definition}[Transition kernel on $S$]\label{Def for p_u}  For any measure $\mu$ on $V_S$, define
	\begin{align*}
		\pp_{\mu}(z_1,z_2;\delta):=
			\left\{
			\begin{aligned}
			&\frac{{\bm 1}_{\{\mu|_E(z_2|_E)=0 \}}+\delta {\bm 1}_{\{\mu|_E(z_2|_E)>0 \}}}{\sum_{z\leftarrow z_1}{\bm 1}_{\{\mu|_E(z|_E)=0 \}}+\delta {\bm 1}_{\{\mu|_E(z|_E)>0 \}}}, \ &z_1\to z_2,\\
			&0, \ &{\rm otherwise,}
			\end{aligned}\right.
			\end{align*}
where $\mu|_E$ is the projective measure of $\mu$ to edges in $E$, and $z|_E$ stands for the edge given by deleting the direction of $z$.
\end{definition}
	
We emphasize that for fixed $\delta$, $\pp_\mu$ only depends the support of $\mu$, i.e., $\pp_\mu=\pp_\nu$ if $\text{supp}(\mu)=\text{supp}(\nu)$. Sometimes, we write $\pp_\mu(z_1,z_2;\delta)$ as $\pp_{\mu}(z_1,z_2)$ for fixed $\delta$.

	For every $E'\subseteq E$, we denote the indicator by $\rho_{E'}$ on $V_S$, where
	\[
	\rho_{E'}(\{z\})=
	1 \text{ if }z|_E\in E', \text{ and }
	0 \text{ if }z|_E\notin E'.
	\]
For convenience, we write
\begin{equation}\label{p_E}
\pp_{E'}(z_1,z_2;\delta):=\pp_{\rho_{E'}}(z_1,z_2;\delta).
\end{equation}

For every fixed path $\omega=(X_n)_{n\geq 0}$, let	
$
 {E}_n^\omega=\left\{e\in E:\ \exists m\le n,\ \mathcal{Z}_m|_E(\omega)=e\right\}
$
be the collection of edges traversed by $X$ up to time $n\geq 1$. We may observe that for any $n\geq 1,$
$\pp_{{E}_n^\omega}$ is the transition probability from $\mathcal{Z}_n$ to $\mathcal{Z}_{n+1}$. Precisely,
	\begin{equation}
	\P(\mathcal{Z}_{n+1}=z_2\,|\,\mathcal{Z}_n=z_1,\mathscr{F}_n)=\pp_{{E}_n^\omega}(z_1,z_2), \text{ for }z_1\to z_2, \label{transition probability}
	\end{equation}
where $\mathscr{F}_n$ is the natural filtration generated by the history of $X$ up to time $n$, i.e., generated by the history of $\mathcal{Z}$ up to time $n-1$.

Set
\begin{eqnarray}\label{eq-empirical-lifted}	
\mathcal{L}^n(A)=\frac{1}{n}\sum_{i=0}^{n-1}\dd_{\mathcal{Z}_i}(A),\ A\subseteq S,\ n\geq 1
\end{eqnarray}
to be the empirical measures of $(\mathcal{Z}_n)_{n\geq 0}$, where $\dd_z$ is the Dirac measure on $V_S$ at $z$. Specifically, $\mathcal{L}^0\equiv0$.  Since  $\pp_{\mathcal{L}^{n}}=\pp_{E_n^\omega}$, the process $(\mathcal{Z}_n,\mathcal{L}^n)_{n\geq 0}$ is a non-homogeneous Markov process with the transition probability
\begin{equation}\label{Z_n,L_n}
  \P\left(\left. \mathcal{Z}_{n+1}=z',\mathcal{L}^{n+1}=\mu'\,\right|\,\mathcal{Z}_n=z,\mathcal{L}^n=\mu\right)=p_{\mu+\dd_z}\left(z,z'\right)\cdot \dd_{\frac{n\mu+\dd_z}{n+1}}\left(\mu'\right),\ n\geq 1,
\end{equation}
where $\dd_{\nu}$ is the Dirac measure on $\mathscr{P}(V_S)$\label{mathscr{P}(T)} at $\nu$.

Denote by $\mathscr{E}$\label{mathscr{E}} the collection of all sequences of subsets $\{E_k\}_{1\le k\le \b}$ satisfying the following conditions: 
\begin{itemize}
	\item[\bf{A1}]$E_k\subset E_{k+1}\subseteq E,\ 1\leq k<\b$.
	\item[\bf{A2}]$|E_k|=k,\ 1\leq k\leq \b$.
	\item[\bf{A3}]There exists $v\in V$ such that $\{x,v\}\in E_1$.
	\item[\bf{A4}]The unique edge in $E_{k+1} \setminus E_k$ is adjacent to some edge in $E_k$.
\end{itemize}
Each $E_k$ induces a connected graph on all its vertices. Intuitively, $\mathscr{E}$ collects all possible sequences of edge sets generated by $(X_nX_{n+1})_{n\geq 0}$. Mathematically speaking, for all fixed paths $\omega=(X_n)_{n\geq 0}$, we have 
$E_k(\omega)=\{X_{n-1}X_{n}:\ n\le{\tau}_{k}\}$,
where
\begin{equation}
	{\tau}_k=\inf\left\{j>{\tau}_{k-1}:\ X_{j-1}X_j\notin \{X_{i-1}X_i:i<j \} \right\}\ \mbox{and}\ {\tau}_1:=0.\label{def of renewal time}
\end{equation}

For all non-empty decreasing subsets $\mathscr{S}_0$\label{mathscr{S}_0} of $\mathscr{S}$, and all $\mu\in\mathscr{P}(V_S),$ set
	\begin{align}
		\mathscr{A}(\mu,\mathscr{S}_0)=\Bigg\{& (\mu_k,r_k,E_k)_{1\leq k\leq \b}:\ \{E_j\}_{1\leq j\leq \b}\in\mathscr{E};
		r_k\ge 0,\ r_l=0\ \text{for}\ E_l\notin\mathscr{S}_0,\ \sum_{k=1}^{\b}r_k=1;\nonumber\\
	&\hskip 5mm \mu_k\in\mathscr{P}(V_S),\ {\rm{supp}}(\mu_k|_E)\subseteq E_k,\ \sum_{k=1}^{\b} r_k \mu_k=\mu\Bigg\}. \label{A}
	\end{align}
 Let $\mathscr{T}_S$\label{mathscr{T}_S} be the collection of all transition probabilities of nearest-neighbour Markov chains on $V_S$.

\begin{definition}\label{Def Lambda} For any $\mu\in\mathscr{P}(V_S)$ and non-empty decreasing subset $\mathscr{S}_0$ of $\mathscr{S},$ let
		\begin{align}
		{\Lambda}_{\delta,\mathscr{S}_0}(\mu)=\inf_{\stackrel{(\mu_k,r_k,E_k)_{k}\in\mathscr{A}(\mu,\mathscr{S}_0)}
        {\qq_{k}\in\mathscr{T}_S:\,\mu_{k}\qq_{k}=\mu_{k}}}
		\sum_{k=1}^{\b}r_k\int_{V_S}R(\qq_{k}\|\pp_{E_k})\ {\rm d}\mu_{k}.\label{rate function}
		\end{align}
Specifically, write $\Lambda_{\delta,\mathscr{S}}(\cdot)$ as $\Lambda_\delta(\cdot)$.
\end{definition}
	
Denote by
\begin{equation}\label{eq-renew-subset}
\mathscr{E}_z:=\{\{E_k\}_{1\le k\le \b}\in\mathscr{E}:\ E_1=\{z|_E\} \},
\end{equation}
the renewal subsets of $E$ starting from edge set $\{z|_E\}$ for any $z\in V_S.$ Set
\begin{equation}
 \mathscr{A}_z(\mu,\mathscr{S}_0)=\left\{(\mu_{k},r_k,E_k)_{1\leq k\leq \b}\in\mathscr{A}(\mu,\mathscr{S}_0):\ \{E_k\}_{1\le k\le \b}\in \mathscr{E}_z \right\}.
 \label{A_z}
\end{equation}
Define
\begin{eqnarray}\label{eq-closed-subset}
\mathcal{C}(\mathscr{S}_0):=\left\{\mu\in\mathscr{P}(V_S):\ \text{supp}(\mu|_E)\subseteq E'\ \text{for some}\ E'\in \mathscr{S}_0 \right\}.
\end{eqnarray}
Then $\mathcal{C}(\mathscr{S}_0)$ is a closed subset of $\mathscr{P}(V_S)$ under the weak convergence topology. 

For any bounded and continuous function $h$, set
\begin{equation}
W^n_{\mathscr{S}_0}(z):=\left\{\begin{array}{ll}
-\frac{1}{n}\log{{\E}_z\left\{\exp\left[-nh(\mathcal{L}^n)\right]\mathbf{1}_{\{\mathcal{L}^n\in\mathcal{C}(\mathscr{S}_0)\}}\right\}},&    {\P}_z(\mathcal{L}^n\in\mathcal{C}(\mathscr{S}_0))>0\\
\infty,&  {\P}_z(\mathcal{L}^n\in\mathcal{C}(\mathscr{S}_0))=0    \end{array}\right.. \label{eq-W_S0^n}
\end{equation}
We name $W^n_{\mathscr{S}_0}(z)$ the restricted Laplace functional; and now present a generalized Laplace principle of $W^n_{\mathscr{S}_0}(z)$ in \cite{HLX2022}, which is a pivotal theorem in this paper.
		
	\begin{theorem}[\cite{HLX2022}, Theorem 2.7]\label{estimate for rate}
For any bounded and continuous function $h$ on $\mathscr{P}(V_S),$
		\begin{equation}
		\lim_{n\to\infty}W^n_{\mathscr{S}_0}(z) = \inf_{\mu\in \mathcal{C}(\mathscr{S}_0)}\left\{\Lambda^z_{\delta,\mathscr{S}_0} (\mu)+h(\mu)\right\},\label{our variational representation}
		\end{equation}
		where
		\begin{align}
		{\Lambda}^z_{\delta,\mathscr{S}_0}(\mu)=\inf_{\stackrel{(\mu_k,r_k,E_k)_{k}\in\mathscr{A}_z(\mu,\mathscr{S}_0)}{\qq_{k}\in\mathscr{T}_S:\, \mu_{k}\qq_{k}=\mu_{k}}}
		\sum_{k=1}^{\b}r_k\int_{V_S}R(\qq_{k}\|\pp_{E_k})\ {\rm d}\mu_{k}.\label{Lambda^z}
		\end{align}
	\end{theorem}

\begin{lemma}[\cite{HLX2022}, Lemma 3.6]\label{inf attainment}
	For fixed $\mu\in\mathscr{P}(V_S)$ and $E'\subseteq E$, set $\nu=T(\mu)$, where $T(\mu)(u)=\mu(\{z:z^-=u\})$ for any $\mu\in\mathscr{P}(V_S)$ and $u\in V$. Then
	\[
	\inf_{\qq\in\mathscr{T}_S:\,\mu \qq=\mu}\int_{V_S}R(\qq\|\pp_{E'})\ {\rm d}\mu=\inf_{\hat{q}\in\mathscr{T}_G:\,\forall z\in V_S, \nu(z^-)\hat{q}(z^-,z^+)=\mu(z)}\int_{V}R(\hat{q}\|\hat{p}_{E'})\ {\rm d}\nu.
	\]
\end{lemma}

By Theorem \ref{estimate for rate}, we deduce the following theorem (see the proof in Section \ref{sec 4.1}).

	\begin{theorem}\label{exp decay}
$\mathbb{P}_{x_0}\left(\mathcal{T}_{\mathscr{S}_0}>n\right)$ decays exponentially with the rate $\inf\limits_{\mu\in \mathcal{C}(\mathscr{S}_0)}{\Lambda}_{\delta,\mathscr{S}_0}(\mu)\in (0,\infty)$:
		\begin{equation}\label{exp-decay 1}
		\lim_{n\to\infty}\frac{1}{n}\log \mathbb{P}_{x_0}\left(\mathcal{T}_{\mathscr{S}_0}>n\right)=-\inf_{\mu\in\mathcal{C}(\mathscr{S}_0)}{\Lambda}_{\delta,\mathscr{S}_0}(\mu).
		\end{equation}
	\end{theorem}

\section{Variational representations of  critical exponent}\label{sec 4.1}
\noindent
An alternative perspective is that Theorem \ref{m-exp decay} presents a concise variational formula (\ref{new exp. of alpha_c})  of  $\alpha_c(\delta)$, the critical exponent for the exponential integrability of $\mathcal{T}_{\mathscr{S}_0}$ through the estimate of the tail probability $\mathbb{P}_{x_0}\{\mathcal{T}_{\mathscr{S}_0}>n\}$.

In order to show Theorem \ref{m-exp decay}, we derive the asymptotic estimate of $\frac{1}{n}\log \mathbb{P}_{x_0}\{\mathcal{T}_{\mathscr{S}_0}>n\}$ in Theorem \ref{exp decay}, a corollary of Theorem \ref{estimate for rate} on lifted directed graphs.  Furthermore, we simplify the variational representation of $\alpha_c(\delta)$ in Theorem \ref{exp decay} through Proposition \ref{another expression of alpha_c}, offering a more convenient approach to demonstrate   Theorem \ref{m-exp decay}.
	Additionally, by Theorem \ref{m-exp decay} (a) and Theorem \ref{exp decay}, we can express $\alpha_c$ as 
	\begin{equation}
	\alpha_c(\delta)=\inf_{\mu\in\mathcal{C}(\mathscr{S}_0)}{\Lambda}_{\delta,\mathscr{S}_0}(\mu), \label{eq-representation-alpha-c} 
	\end{equation}
	a formulation that will be utilized in Section \ref{sec 4.2} and Appendices \ref{sec-pf-prop-4.1}, \ref{sec-pf-prop-4.2}.

 At first, we prove Theorem  \ref{exp decay}.

 \begin{proof}[Proof of Theorem \ref{exp decay}]
Noting that $\{\mathcal{L}^n\in\mathcal{C}(\mathscr{S}_0)\}=\{\mathcal{T}_{\mathscr{S}_0}>n\}$,
	it suffices to show that for any $z\in V_S$ with $z^-=x,$
\[
\lim_{n\to\infty}\frac{1}{n}\log \mathbb{P}_{x_0}\left(\left.\mathcal{L}^n\in\mathcal{C}(\mathscr{S}_0)\right\vert \mathcal{Z}_0=z \right)=-\inf_{\mu\in\mathcal{C}(\mathscr{S}_0)}{\Lambda}^z_{\delta,\mathscr{S}_0}(\mu),
\]
since 
$\inf_{\mu\in\mathcal{C}(\mathscr{S}_0)}{\Lambda}_{\delta,\mathscr{S}_0}(\mu)=\min_{z:\,z^-=x}\inf_{\mu\in\mathcal{C}(\mathscr{S}_0)}{\Lambda}^z_{\delta,\mathscr{S}_0}(\mu),$
and
\[\lim_{n\to\infty}\frac{1}{n}\log \mathbb{P}_{x_0}\left(\mathcal{L}^n\in\mathcal{C}(\mathscr{S}_0)\right)=\max_{z:\,z^-=x}\lim_{n\to\infty}\frac{1}{n}\log \mathbb{P}_{x_0}\left(\left.\mathcal{L}^n\in\mathcal{C}(\mathscr{S}_0)\right\vert \mathcal{Z}_0=z\right).\]

Let $h_0$ be some bounded continuous function vanishing on $\mathcal{C}(\mathscr{S}_0)$. Then
\begin{align}\label{(4.1.1)}
\lim_{n\to\infty}\frac{1}{n}\log \mathbb{E}_{x_0}\left[\left. e^{-nh_0(\mathcal{L}^n)}\textbf{1}_{\{\mathcal{L}^n\in\mathcal{C}(\mathscr{S}_0)\}}\right\vert \mathcal{Z}_0=z\right]=\lim_{n\to\infty}\frac{1}{n}\log \mathbb{P}_{x_0}\left(\left.\mathcal{L}^n\in\mathcal{C}(\mathscr{S}_0)\right\vert \mathcal{Z}_0=z\right).
\end{align}

Applying  Theorem \ref{estimate for rate} to $h_0$, we obtain
\begin{equation}\label{(4.1.3)}
\lim_{n\to\infty}\frac{1}{n}\log \mathbb{E}_{x_0}\left[\left. e^{-nh_0(\mathcal{L}^n)}\textbf{1}_{\{\mathcal{L}^n\in\mathcal{C}(\mathscr{S}_0)\}}\right\vert \mathcal{Z}_0=z\right]
=\inf_{\mu\in\mathcal{C}(\mathscr{S}_0)}\Lambda_{\delta,\mathscr{S}_0}^z(\mu).
\end{equation}
By (\ref{(4.1.1)}) and (\ref{(4.1.3)}), we complete the proof.
\end{proof}

\begin{proposition}\label{another expression of alpha_c}
 $\inf\limits_{\mu\in\mathcal{C}(\mathscr{S}_0)}\Lambda_{\delta,\mathscr{S}_0}(\mu)$ can also be expressed as			
\begin{align}\label{exprs1 alpha_c}
	\inf_{\mu\in\mathcal{C}(\mathscr{S}_0)}\Lambda_{\delta,\mathscr{S}_0}(\mu)=&\inf_{\stackrel{\left(\mu,E_0,q\right):E_0\in\mathscr{S}_0,\qq\in\mathscr{T}_S,}{\mu\in\mathscr{P}(V_S), \text{\rm supp}(\mu|_E)\subseteq E_0,\mu \qq=\mu}}\int_{V_S} R(\qq\|\pp_{E_0})\ {\rm d}\mu. 
\end{align}	
\end{proposition}

We will show this proposition in Appendix \ref{sec-pf-prop-3.1}.  Now, it is time to prove  Theorem \ref{m-exp decay}.

  \begin{proof}[Proof of Theorem \ref{m-exp decay} (a)]  Due to Theorem \ref{exp decay} and Proposition \ref{another expression of alpha_c}, we only need to show
  \begin{equation}\label{exprss1-1 alpha_c}
     \inf_{\stackrel{\left(\mu,E_0,q\right):E_0\in\mathscr{S}_0,\qq\in\mathscr{T}_S,}{\mu\in\mathscr{P}(V_S), \text{\rm supp}(\mu|_E)\subseteq E_0,\mu \qq=\mu}}\int_{V_S} R(\qq\|\pp_{E_0})\ {\rm d}\mu= \inf_{\stackrel{\left(\nu,E_0,\hat{q}\right): E_0\in\mathscr{S}_0,\ \hat{q}\in\mathscr{T}(E_0)}{\nu\in\mathscr{P}(V),\ \nu \hat{q}=\nu,\ {\rm supp}(\nu)\subseteq V_0 }}\int_{V} R(\hat{q}\|\hat{p}_{E_0})\ {\rm d}\nu.
   \end{equation}
  By Lemma \ref{inf attainment}, it is sufficient to verify that the right hand side (RHS) of \eqref{exprss1-1 alpha_c} equals
       \begin{equation}\label{exprss1-2 alpha_c}
    \inf_{\stackrel{\left(\mu,E_0,\hat{q}\right): {\rm supp}(\mu|_E)\subseteq   E_0\in\mathscr{S}_0}{T(\mu)(z^-)\hat{q}(z^-,z^+)=\mu(z)}}\int_{V} R(\hat{q}\|\hat{p}_{E_0})\ {\rm d}T(\mu),
      \end{equation}
 where $T(\mu)(u)=\mu(\{z:z^-=u\})$ for any $\mu\in\mathscr{P}(V_S)$ and $u\in V$.

Let
$$\mathcal{I}_{\rm nf}=\left\{\left(\nu,E_0,\hat{q}\right):\ E_0\in\mathscr{S}_0,\ \hat{q}\in\mathscr{T}(E_0),\ \nu\in\mathscr{P}(V),\ \nu \hat{q}=\nu,\ {\rm supp}(\nu)\subseteq V_0\right\},$$
$$\mathcal{I}_{\rm nf}^T=\left\{\left(\mu,E_0,\hat{q}\right):\ {\rm supp}(\mu|_E)\subseteq   E_0\in\mathscr{S}_0,\ T(\mu)(z^-)\hat{q}(z^-,z^+)=\mu(z),\ z\in V_S\right\}.$$

Given any $\left(\mu,E_0,\hat{q}\right)\in \mathcal{I}_{\rm nf}^T.$ Let $\nu=T(\mu)$. Summing $\nu(z^-)\hat{q}(z^-,z^+)=\mu(z)$ over all $z$ with $z^+=y$ for some $y\in V$, we get that $\nu\hat{q}=\nu$, and
$$``\hat{q}\in\mathscr{T}(E_0),\ {\rm supp}(\nu)\subseteq V_0"\Longleftrightarrow ``\mbox{for}\ x,y\in V,\ \nu(x)\hat{q}(x,y)>0\Rightarrow xy\in E_0".$$
However, for any $x,y\in V$ with $\nu(x)\hat{q}(x,y)>0,$ since
$\nu(x)\hat{q}(x,y)=\mu(z)\ \mbox{for}\ z\ \mbox{with}\ z^-=x,z^+=y,$
we have $z\in{\rm supp}(\mu)$ and $xy\in E_0$. Hence, $\hat{q}\in\mathscr{T}(E_0),\ {\rm supp}(\nu)\subseteq V_0$ and $\left(\nu,E_0,\hat{q}\right)\in \mathcal{I}_{\rm nf}.$ This implies that the RHS of \eqref{exprss1-1 alpha_c} is no more than \eqref{exprss1-2 alpha_c}.

On the other hand, for any $\left(\nu,E_0,\hat{q}\right)\in \mathcal{I}_{\rm nf},$ set $\mu(z)=\nu(z^-)\hat{q}(z^-,z^+),\ z\in V_S$. We verify that $\nu=T(\mu)$ and $\left(\mu,E_0,\hat{q}\right)\in \mathcal{I}_{\rm nf}^T$ as follows:
\begin{itemize}
	\item $\nu=T(\mu)$: For any $x\in V,$ summing $\mu(z)=\nu(z^-)\hat{q}(z^-,z^+)$ over all $z$ with $z^-=x$, by $\sum\limits_{y\in V}\hat{q}(x,y)=1$, we obtain $\nu(x)=T(\mu)(x)$.
	\item ${\rm supp}(\mu|_E)\subseteq E_0$: Since for any $z\in V_S$, $\mu(z)=\nu(z^-)\hat{q}(z^-,z^+)>0$ implies $z|_E\in E_0$, we show that ${\rm supp}(\mu|_E)\subseteq E_0$.
\end{itemize}
Thus the RHS of (\ref{exprss1-1 alpha_c}) is no less than \eqref{exprss1-2 alpha_c}. So far we have finished the proof.\end{proof}

  \begin{proof}[Proof of Theorem \ref{m-exp decay} (b)]  Our strategy is to demonstrate that there is a $p\in (0,1)$ such that
\begin{equation}\label{more alpha_c-2}
0<\liminf_{n\to\infty}\frac{\mathbb{P}_{x_0} (\mathcal{T}_{\mathscr{S}_0}>n)}{p^n}\le \limsup_{n\to\infty}\frac{\mathbb{P}_{x_0} (\mathcal{T}_{\mathscr{S}_0}>n)}{p^n}<\infty.
\end{equation}
This implies that
$
\lim_{n\to \infty}\frac{1}{n}\log\mathbb{P}_{x_0}\left(\mathcal{T}_{\mathscr{S}_0}>n\right)=-\log p.
$
By part (a) in Theorem  \ref{m-exp decay}, $p=e^{-\alpha_c(\delta)}$.

For any two nonnegative sequences $\{a_n\}_{n\ge 0}$ and $\{b_n\}_{n\ge0}$, call they have the same order if for some positive constants $h_1$ and $h_2$,
\begin{equation}
h_1b_n\leq a_n\leq h_2b_n,\ n\geq 0,\label{Theta()}
\end{equation}
and denote this by $a_n=\Theta(b_n)$. To prove \eqref{more alpha_c-2}, we need to show the following claim:
\vskip1mm

{  \emph
	{For $(\mathcal{Z}_n)_{n\ge0}$ with initial condition $\mathcal{Z}_0^-=x$ (choose $\mathcal{Z}_0^+$ uniformly among those vertices in $G$ adjacent to $x$), there exists a constant $p\in(0,1)$ such that
		\begin{equation}\label{exp-crt-clm}
		\mathbb{P}_{x_0}(\mathcal{T}_{\mathscr{S}_0}>n)=\Theta(p^n).
		\end{equation}}}

To this end, we have to employ the decomposition of $(\mathcal{Z}_n)_{n\ge 0}$ again.

Recall the renewal times ${\tau}_i$ defined in \eqref{def of renewal time}, choose some sequence $\{E_j\}_{1\leq j\leq d}\in\mathscr{E}$. Set $l=\min\{k\le \b: E_k\notin\mathscr{S}_0 \}$, $l$ depends on $\{E_j\}_j$. Note that under condition $\mathcal{Z}_{\tau_k}|_E\in E_k\ (1\le k\le i)$, the process $(\mathcal{Z}_n)_{n\ge \tau_i}$ can be regarded as a Markov process on $\{\overrightarrow{uv}: uv\in E_{i+1}\}$ when it stays in the subgraph $\{\overrightarrow{uv}: uv\in E_{i}\}$.   By the total probability formula,
we have
\begin{align}
\mathbb{P}_{x_0}\left(\mathcal{T}_{\mathscr{S}_0}>n\right)&=\sum_{\{E_j\}_j\in\mathscr{E}}\sum_{n_1+\dots+n_{l-1}>n}\prod_{i=1}^{l-1}\mathbb{P}_{x_0}\left({\tau}_{i}-{\tau}_{i-1}=n_i, \mathcal{Z}_{\tau_i}|_E\in E_i \Big|\mathcal{Z}_{\tau_k}|_E\in E_k,1\le k\le i-1\right).  \label{exp-crt-clm1}
\end{align}

Hence, a crucial aspect is to estimate the tail probability of $\tau_i-\tau_{i-1}$ under condition

\noindent $\{\mathcal{Z}_{\tau_k}|_E\in E_k,1\le k\le i-1\}$.  To address this,  we introduce a result on Markov chains with a finite state space as follows.
\begin{lemma}\label{order of Markov process}
	For an irreducible Markov chain $\{\xi_n\}_{n\ge0}$ on a finite state space $\widetilde{S}$, the probability it does not visit some subspace $A\neq \widetilde{S}$ until time $n$ has the same order as $p_0^n$ for some $p_0\in(0,1)$, i.e.,
	$
	\P(\tau> n)=\Theta(p_0^n),
	$
	where $\tau:=\inf\{n: \xi_n\in A \}$.
\end{lemma}

We prove Lemma \ref{order of Markov process} in Appendix \ref{apdx-mrkproc}. Applying Lemma \ref{order of Markov process}, we obtain
\begin{equation}
\mathbb{P}_{x_0}\left({\tau}_{i}-{\tau}_{i-1}>n_i\Big|\mathcal{Z}_{\tau_k}|_E\in E_k,1\le k\le i-1\right)=\Theta\left(p_{i,\{E_j\}}^{n_i}\right)\label{order in subgraph},
\end{equation}
see Figure \ref{schedule line 3}. Note that
\begin{align}
&\mathbb{P}_{x_0} \left( \tau_i-\tau_{i-1}>n_i, \mathcal{Z}_{\tau_i}|_E\in E_i \left| \mathcal{Z}_{\tau_k}|_E\in E_k,1\le k\le i-1 \right. \right)\nonumber\\
&=\mathbb{P}_{x_0} \left( \tau_i-\tau_{i-1}>n_i \left| \mathcal{Z}_{\tau_k}|_E\in E_k,1\le k\le i-1 \right. \right) \times\nonumber\\
&\hskip 0.6cm \sum_{z:z|_E\in E_{i-1}}\Big[\mathbb{P}_{x_0} \left( \mathcal{Z}_{\tau_i}|_E\in E_i \left| \mathcal{Z}_{n_i}=z,  \tau_i-\tau_{i-1}>n_i, \mathcal{Z}_{\tau_k}|_E\in E_k,1\le k\le i-1 \right. \right)\times \nonumber\\
&\hskip 1.8cm \mathbb{P}_{x_0} \left(  \mathcal{Z}_{n_i}=z \left| \tau_i-\tau_{i-1}>n_i, \mathcal{Z}_{\tau_k}|_E\in E_k,1\le k\le i-1 \right. \right) \Big] . \label{eq-order-estimate}
\end{align}

\begin{figure}[htbp]
	\centering{\includegraphics{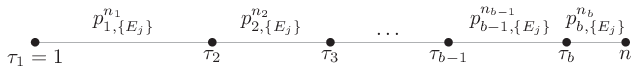}}
	\caption{\small{Under condition $\mathcal{Z}_{\tau_k}|_E\in E_k\ (1\le k\le i-1)$, the probability of the event $\{\tau_{i}-\tau_{i-1}>n_i\}$  has the same order as $p_{i,\{E_j\}}^{n_i}$.}}
	\label{schedule line 3}
\end{figure}

Set $(\tilde{\mathcal{Z}}_n)_{n\ge0}$ to be a Markov process with the state space $\{\overrightarrow{uv}: uv\in \tilde{E}_{i}\}$ and the transition probability $\pp_{E_{i-1}}$ (see definition in \eqref{p_E}), where $\tilde{E}_i$ is a union of $E_{i-1}$ and all edges adjacent to $E_{i-1}$. We observe
\begin{equation}
\tilde{\mathbb{P}}_z(\tilde{\mathcal{Z}}_{\tilde{\tau}_i}|_E\in E_i)=\mathbb{P}_{x_0} \left( \mathcal{Z}_{\tau_i}|_E\in E_i \left| \mathcal{Z}_{n_i}=z,  \tau_i-\tau_{i-1}>n_i, \mathcal{Z}_{\tau_k}|_E\in E_k,1\le k\le i-1 \right. \right), \label{eq-order-representation}
\end{equation}
where $\tilde{\mathbb{P}}_z$ is the probability derived by $(\tilde{\mathcal{Z}}_n)_{n\ge0}$ with starting point $z$, and $\tilde{\tau}_i:=\inf\{n:\tilde{\mathcal{Z}}_n|_E\notin E_{i-1}  \}$ is the stopping time such that $\tilde{\mathcal{Z}}_n$ runs out of $\{\overrightarrow{uv}: uv\in E_{i-1}\}$.
Actually, \eqref{eq-order-representation} is a constant in $(0,1)$, and the amount of $z$, i.e., $ |\{ z: z|_E\in E_{i-1}\}|$, is finite. Therefore, \eqref{eq-order-estimate} has the same order as \eqref{order in subgraph}.

Now, we are focusing on proving (\ref{exp-crt-clm}).

 Noting that \eqref{eq-order-estimate} has the same order as (\ref{order in subgraph}), we deduce
\begin{align*}
(\ref{exp-crt-clm1})&\ge\sum_{\{E_j\}_j\in\mathscr{E}}\max_{1\le i\le l-1}\left\{\mathbb{P}_{x_0}\left({\tau}_{i}-{\tau}_{i-1}>n\Big|\mathcal{Z}_{\tau_k}|_E\in E_k,1\le k\le \b\right)\right\}
=\Theta\left(\max_{1\le i\le l-1,\{E_j\}_j\in\mathscr{E}}\{p_{i,\{E_j\}}\}^n\right).
\end{align*}
And,
\begin{align*}
(\ref{exp-crt-clm1})
&\le\sum_{\{E_j\}_j\in\mathscr{E}}
 \sum_{n_1+\dots+n_{l-1}=n}\prod_{i=1}^{l-1}\mathbb{P}_{x_0}\left({\tau}_{i}-{\tau}_{i-1}>n_i, \mathcal{Z}_{\tau_i}\in E_i \Big|\mathcal{Z}_{\tau_k}|_E\in E_k,1\le k\le i-1\right)\\
&= \Theta\left(\sum_{\{E_j\}_j\in\mathscr{E}}\sum_{n_1+\dots+n_{l-1}=n}\prod_{i=1}^{l-1}p_{i,\{E_j\}}^{n_i}\right).
\end{align*}
For fixed $\{E_j\}_j\in\mathscr{E}$, let $k'=k'(\{E_j\})$ be an $i$ such that $p_{i,\{E_j\}}=\max_{1\le k\le l-1}\{p_{k,\{E_j\}}\}.$ Then
\begin{align}
&\sum_{\{E_j\}_j\in\mathscr{E}}\sum_{n_1+\dots+n_{l-1}=n}\prod_{i=1}^{l-1}p_{i,\{E_j\}}^{n_i}\nonumber\\
&=\sum_{\{E_j\}_j\in\mathscr{E}}(p_{k',\{E_j\}})^{n}\cdot\sum_{n_1+\dots+n_{k'-1}+n_{k'+1}+\dots+n_{l-1}\le n}\prod_{1\le i\not=k'\le l-1}\left(\frac{p_{i,\{E_j\}}}{p_{k',\{E_j\}}}\right)^{n_i}.\label{(4.9)}
\end{align}
On the one hand,
\begin{equation*}
(\ref{(4.9)})\ge\sum_{\{E_j\}_j\in\mathscr{E}}(p_{k',\{E_j\}})^n\cdot\prod_{1\le k\not=k'\le l-1}\sum_{r=0}^{\lfloor\frac{n}{l}\rfloor}\left(\frac{p_{k,\{E_j\}}}{p_{k',\{E_j\}}}\right)^r=\Theta\left(\max_{1\le i\le l-1,\{E_j\}_j\in\mathscr{E}}\{p_{i,\{E_j\}}\}^n\right),
\end{equation*}
where for any $a\in\mathbb{R},$ $\lfloor a\rfloor$ is the biggest integer no more than $a$. On the other hand,
\begin{align*}
(\ref{(4.9)})\le\prod_{1\le k\not=k'\le l-1}\sum_{r=0}^{n}\left(\frac{p_{k,\{E_j\}}}{p_{k',\{E_j\}}}\right)^r
=\Theta\left(\max_{1\le i\le l-1,\{E_j\}_j\in\mathscr{E}}\{p_{i,\{E_j\}}\}^n\right).
\end{align*}
Therefore, $(\ref{(4.9)})=\Theta\left(\max_{1\le i\le l-1,\{E_j\}_j\in\mathscr{E}}\{p_{i,\{E_j\}}\}^n\right)$, which implies \[\mathbb{P}_{x_0}\left(\mathcal{T}_{\mathscr{S}_0}>n\right)\le\Theta\left(\max_{1\le i\le l-1,\{E_j\}_j\in\mathscr{E}}\{p_{i,\{E_j\}}\}^n\right).\]

To sum up, choose
$p=\max_{1\le i\le l-1,\{E_j\}_j\in\mathscr{E}}\{p_{i,\{E_j\}}\}.$
Then $\mathbb{P}_{x_0}\left(\mathcal{T}_{\mathscr{S}_0}>n\right)=\Theta\left(p^n\right)$. \end{proof}

  \begin{proof}[Proof of Theorem \ref{m-exp decay} (c)]  Due to  $\mathbb{P}_{x_0}\left(\mathcal{L}^n\in\mathcal{C}(\mathscr{S}_0)\right)=\mathbb{P}_{x_0}\left(\mathcal{T}_{\mathscr{S}_0}>n\right)$, Part (a) in Theorem \ref{m-exp decay} shows that
	\begin{equation}
	\lim_{n\to\infty}\frac{1}{n}\log\mathbb{P}_{x_0}\left(\mathcal{T}_{\mathscr{S}_0}>n\right)=-\alpha_c(\delta).\label{exponential decay}
	\end{equation}
Note that
	\begin{align}
	\mathbb{E}_{x_0}\left[ e^{\alpha \mathcal{T}_{\mathscr{S}_0}}\right]=\sum_{n=0}^\infty  (e^{\alpha}-1)\cdot e^{\alpha n} \mathbb{P}_{x_0}\left(\mathcal{T}_{\mathscr{S}_0}>n\right).\label{equ in proof of exponential decay}
	\end{align}

\noindent	{\bf\ \emph{Case 1}}.	  If $\alpha<\alpha_c(\delta)$, then choose some $0<\varepsilon<\alpha_c(\delta)-\alpha$, by (\ref{exponential decay}) there exists some $N>0$ such that for all $n>N$, $\mathbb{P}_{x_0}\left(\mathcal{T}_{\mathscr{S}_0}>n\right)<e^{-n(\alpha_c(\delta)-\varepsilon)}$, which implies
	$
	e^{\alpha n} \mathbb{P}_{x_0}\left(\mathcal{T}_{\mathscr{S}_0}>n\right)<e^{n(\alpha+\varepsilon-\alpha_c(\delta))}.
	$
Since $\alpha+\varepsilon-\alpha_c(\delta)<0$, by (\ref{equ in proof of exponential decay}), we have that $\mathbb{E}_{x_0}\left[ e^{\alpha \mathcal{T}_{\mathscr{S}_0}}\right]<\infty$.

 \noindent	{\bf\ \emph{Case 2}}.	  If $\alpha>\alpha_c(\delta)$, then choose some $0<\varepsilon<\alpha-\alpha_c(\delta)$, and similarly we can find some $N>0$ such that for all $n>N$,
	$
	e^{\alpha n} \mathbb{P}_{x_0}\left(\mathcal{T}_{\mathscr{S}_0}>n\right)>e^{n(\alpha-\varepsilon-\alpha_c(\delta))}.
	$
Due to  $\alpha-\varepsilon-\alpha_c(\delta)>0$, by (\ref{equ in proof of exponential decay}), we see that $\mathbb{E}_{x_0}\left[ e^{\alpha \mathcal{T}_{\mathscr{S}_0}}\right]=\infty$.

 \noindent	{\bf\ \emph{Case 3}}.	 If $\alpha=\alpha_c(\delta)$,  Part (c)  in Theorem \ref{m-exp decay} shows that there is some constant $c>0$ and sufficiently large $M>0$ such that for all $n>M$,
 \begin{equation}\label{more alpha_c-1}
     c<e^{n\alpha_c(\delta)}\mathbb{P}_{x_0}\left(\mathcal{T}_{\mathscr{S}_0}>n\right).
 \end{equation}
  Thus, taking  $\alpha=\alpha_c(\delta)$ in (\ref{equ in proof of exponential decay}), by (\ref{more alpha_c-1}), we obtain
 $\mathbb{E}_{x_0}\left[e^{\alpha_c(\delta)\mathcal{T}_{\mathscr{S}_0}}\right]=\infty$.
 \end{proof}

\section{Analytic properties and  asymptotic behaviors of critical exponent}\label{sec 4.2}
\noindent     In this section, we mainly prove Theorem \ref{thm_prop_alpha_c}. At first, we show how to prove  Theorem   \ref{Mthm C_Edelta} by Theorem \ref{thm_prop_alpha_c}, which presents some  analytic properties and asymptotic behaviors of critical exponent for the exponential integrability of cover time $C_E$. 

{

\begin{proof}[Proof of Theorem \ref{Mthm C_Edelta}]
We may see that for any edge $e$ in star-shaped graphs or triangle, $E_e=E$. By Theorem \ref{thm_prop_alpha_c} (c), we deduce $\lim\limits_{\delta\to 0}\alpha_c^1(\delta)=\infty$. Next we prove that if $\lim\limits_{\delta\to 0}\alpha_c^1(\delta)=\infty$, then the graph should be star-shaped graphs or triangle. By Theorem \ref{thm_prop_alpha_c} (c), this can be deduced by that if $E_e=E$ for any edge $e$ adjacent to the starting point of ORRWs, then the graph should be star-shaped graphs or triangle.

Consider some edge $e$ adjacent to the starting point. Then $E_e=E$. We assume $e=uv$ for some vertices $u,v$.

\noindent \emph{\textbf{Case 1.}}
\emph{If for $u$ or $v$, there is only one edge adjacent to this vertex. Then the graph should be a star-shaped graph.}

\noindent \emph{\textbf{Case 2.}}
\emph{If for both $u$ and $v$, there are two or more edges adjacent. Assume $u_0\sim u$, $v_0\sim v$, where $u_0\neq v$ and $v_0\neq u$. Further, we assume the starting point to be $u$. Consider edge $e'=uu_0$. Since $E_{e'}=E$, $\{v_0,v\}\cap\{u_0,u\}\neq\emptyset$. By $v\neq u,u_0$ and $v_0\neq u$, we obtain that $u_0=v_0$.}

\emph{Assume that there is some other vertex $w\notin\{u,v,u_0,v_0\}$. Noting that $E_e=E$ and that the graph is connected, $w$ is adjacent to $u$ or $v$. If $u\sim w$, by $w\neq u_0,v$, $u\neq u_0,v$, we deduce that $u_0v\notin E_{uw}$, which contradicts to $E_{uw}=E$. Another case is $v\sim w$. By $w\neq u,v_0$, $v\neq u,v_0$, we have that $wv\notin E_{uv_0}$, which contradicts to $E_{uv_0}=E$.}

\emph{Therefore, there is only three vertices in the graph, where each vertex is adjacent to other vertices. This implies that the graph should be a triangle.}

To sum up, we show that the sufficient and necessary condition of $\lim\limits_{\delta\to0}\alpha_c^1(\delta)=\infty$ is that the graph is star-shaped graphs or triangle. \end{proof}
}

\begin{proposition}\label{another expression of alpha_c 1}
	\begin{itemize}
			\item[(a)] 	
	\begin{align}
	\inf_{\mu\in\mathcal{C}(\mathscr{S}_0)}\Lambda_{\delta,\mathscr{S}_0}(\mu)=\inf_{\stackrel{(\mu,E_0):E_0\in\mathscr{S}_0,}{\mu\in\mathscr{P}(V_S) , {\rm supp}(\mu|_E)\subseteq E_0}}  \Bigg\{\Lambda_{1}(\mu)+\int_{\{\overrightarrow{uv}: uv\in \partial E_{0}\}}\log{\frac{d(z)-k(z)+k(z)\delta}{d(z)\delta}}\ {\rm d}\mu\Bigg\},
	\label{exprs2  alpha_c}
	\end{align}
	where $d(z),k(z)$ are the out-degrees of $z$ in the graphs $S$ and $\{\overrightarrow{uv}: uv\in E_{0}\}$ respectively.
	\item[(b)] The infimum in \eqref{exprs2  alpha_c} is attained at some $\mu\in\mathscr{P}(V_S)$ with $\mu|_E\left(\partial E_0\right)>0$.
\end{itemize}
\end{proposition}

\begin{proposition}\label{property of alpha_c 2}
	For $\mathscr{S}_0\neq \mathscr{S}$, $\alpha_c(\delta)$ is strictly decreasing in $\delta>0$. If there exist some $E_0\in \mathscr{S}_0$ and some nonempty set $E'\subseteq E_0$ such that $E'$ is the support of the first marginal measure of some invariant measure on $V_S$ and $E'\cap \partial E_0=\emptyset$, then
	$
	\lim_{\delta\to0}\alpha_c(\delta)=\inf_{(\mu,E_0):\,E_0\in\mathscr{S}_0,\text{\rm supp}(\mu|_E)\subseteq E_0\setminus \partial E_0}\Lambda_1(\mu)<\infty.
	$
	Otherwise, $\lim\limits_{\delta\to 0}\alpha_c(\delta)=\infty$.
\end{proposition}

 We will prove Proposition \ref{another expression of alpha_c 1} in Appendix \ref{sec-pf-prop-4.1} and Proposition \ref{property of alpha_c 2} in Appendix \ref{sec-pf-prop-4.2}. According to these two propositions, we finish the proof of Theorem \ref{thm_prop_alpha_c} as follows.

	\begin{proof}[Proof of Theorem \ref{thm_prop_alpha_c} (a)]

\noindent {\bf \emph{Monotonicity}}. For all $0<\delta_1<\delta_2$,
\begin{align*}	
R\left(\qq(x,\cdot)\|\pp_{E_0}(x,\cdot;\delta_2)\right)-R\left(\qq(x,\cdot)\|\pp_{E_0}(x,\cdot;\delta_1)\right)=\int_{V_S} \log{\frac{\pp_{E_0}(x,y;\delta_1)}{\pp_{E_0}(x,y;\delta_2)}}\ \qq(x,{\rm d}y).	
\end{align*}
Since $\mathcal{C}(\mathscr{S}_0)$ is closed, we may assume that the infimum in the expression of $\alpha_c(\delta_1)$ in \eqref{exprs1 alpha_c} is attained at $\mu_1\in \mathcal{C}(\mathscr{S}_0)$, $E_0\in\mathscr{S}_0$, $\qq_1\in\mathscr{T}_S$, where $\mu_1\qq_1=\mu_1$, and $\text{supp}(\mu|_E)\subseteq E_0$.
Noting $\mu_1(\partial (\{\overrightarrow{uv}: uv\in E_{0}\}))>0$ by Proposition \ref{another expression of alpha_c 1} (b), and that for $x\in\partial ( E_0\times\{1,1\})$ and $y\in \{\overrightarrow{uv}: uv\in E_{0}\}$ with $x\to y$,
\[
\frac{\pp_{E_0}(x,y;\delta_1)}{\pp_{E_0}(x,y;\delta_2)}\le\frac{\delta_1}{\delta_2}\cdot\frac{(\b-1)\delta_2+1}{(\b-1)\delta_1+1}<1.
\]
By Proposition \ref{another expression of alpha_c}, we obtain that
\begin{align}
&\alpha_c(\delta_2)-\alpha_c(\delta_1)\nonumber\\
&\le\int_{V_S}\left[R\left(\qq_1(x,\cdot)\|\pp_{E_0}(x,\cdot;\delta_2)\right)-R\left(\qq_1(x,\cdot)\|\pp_{E_0}(x,\cdot;\delta_1)\right)\right]
\ {\rm d}\mu_1\nonumber\\
&=\int_{V_S}\left[\int_{V_S} \log{\frac{\pp_{E_0}(x,y;\delta_1)}{\pp_{E_0}(x,y;\delta_2)}}\qq_1(x,{\rm d}y)\right]{\rm d}\mu_1\nonumber\\
&\le\log{\left(\frac{\delta_1}{\delta_2}\cdot\frac{(\b-1)\delta_2+1}{(\b-1)\delta_1+1}\right)}\mu_1(\partial (\{\overrightarrow{uv}: uv\in E_{0}\}))<0.
\label{proof of monotonicity of alpha_c}
\end{align}
Namely $\alpha_c(\delta_2)<\alpha_c(\delta_1).$

\vskip 2mm

\noindent {\bf \ \emph{Continuity}}.
When considering continuity, we dominate $\sup\limits_{\mu:\,\Lambda_{1,\mathscr{S}_0}(\mu)<\infty}|\Lambda_{\delta_1,\mathscr{S}_0}(\mu)-\Lambda_{\delta_2,\mathscr{S}_0}(\mu)|$ with some continuous function of $\delta_1,\delta_2$.

Fix some $\mu\in \mathscr{P}_\Lambda(S):=\{\mu\in\mathscr{P}(V_S):\ \Lambda_{1,\mathscr{S}_0}(\mu)<\infty\}$. For $\delta_1<\delta_2$, and  $x\notin\partial (\{\overrightarrow{uv}: uv\in E_{l}\})$,
\[
R(\qq_{l}(x,\cdot)\|\pp_{E_l}(x,\cdot;\delta_2))-R(\qq_{l}(x,\cdot)\|\pp_{E_l}(x,\cdot;\delta_1))=\int_{V_S} \log{\frac{\pp_{E_l}(x,y;\delta_1)}{\pp_{E_l}(x,y;\delta_2)}}\ \qq_{l}(x,{\rm d}y).
\]
Otherwise, we can obtain
$
\frac{\pp_{E_l}(x,y;\delta_1)}{\pp_{E_l}(x,y;\delta_2)}=\frac{k(x)\delta_2+d(x)-k(x)}{\delta_2}\cdot\frac{\delta_1}{k(x)\delta_1+d(x)-k(x)}\ge {\frac{\delta_1}{\delta_2}}.
$
This implies
\[
R(\qq_{l}(x,\cdot)\|\pp_{E_l}(x,\cdot;\delta_2))-R(\qq_{l}(x,\cdot)\|\pp_{E_l}(x,\cdot;\delta_1))\ge \log{\frac{\delta_1}{\delta_2}}.
\]
Choose some $\mu_{k},r_k,E_k,\qq_{k}$ such that
$
\Lambda_{\delta_2,\mathscr{S}_0}(\mu)=\sum_{k=1}^{\b}r_k\int_{V_S}R(\qq_{k}(x,\cdot)\|\pp_{E_k}(x,\cdot;\delta_2))\ {\rm d}\mu_{k}.
$
By the definition of $\Lambda_{\delta_1,\mathscr{S}_0}(\mu)$ and the monotonicity of $\Lambda_\delta$,
\begin{align*}
	0&\ge \Lambda_{\delta_2,\mathscr{S}_0}(\mu)-\Lambda_{\delta_1,\mathscr{S}_0}(\mu)\\
	&\ge \sum_{k=1}^{\b}r_k\int_{S}\left[R(\qq_{k}(x,\cdot)\|\pp_{E_k}(x,\cdot;\delta_2))-R(\qq_{k}(x,\cdot)\|\pp_{E_k}(x,\cdot;\delta_1))\right]\ {\rm d}\mu_{k}
	\ge\log{\frac{\delta_1}{\delta_2}}.
\end{align*}
Since $\mu$ is arbitrary, we demonstrate
$
0\le\sup_{\mu\in\mathscr{P}_\Lambda(S)}|\Lambda_{\delta_1,\mathscr{S}_0}(\mu)-\Lambda_{\delta_2,\mathscr{S}_0}(\mu)|\le\log{\frac{\delta_2}{\delta_1}}.
$
It immediately shows that
$
0\le|\alpha_c(\delta_2)-\alpha_c(\delta_1)|\le \Lambda_{\delta_1,\mathscr{S}_0}(\tilde{\mu})-\Lambda_{\delta_2,\mathscr{S}_0}(\tilde{\mu})\le\log{\frac{\delta_2}{\delta_1}},
$
where $\tilde{\mu}\in\mathcal{C}(\mathscr{S}_0)$, and $\alpha_c(\delta_2)$ is attained at $\tilde{\mu}$ (Since $\mathcal{C}(\mathscr{S}_0)$ is closed, the infimum can be reached at some $\tilde{\mu}\in\mathcal{C}(\mathscr{S}_0)$).
By the squeeze theorem, for any positive constant $\delta_0$ and $\delta_1,\delta_2$ with $\delta_0\le\delta_1<\delta_2$, we have $\lim\limits_{|\delta_2-\delta_1|\to 0}|\alpha_c(\delta_2)-\alpha_c(\delta_1)|=0$.
\end{proof}

\begin{proof}[Proof of Theorem \ref{thm_prop_alpha_c} (b)]

Take an edge $e_0$ containing the starting point $x_0$ with $\{e_0\}\in\mathscr{S}_0$. Recalling (\ref{def of renewal time}) for the definition of renewal times, we obtain
$
\mathbb{P}_{x_0}\left(\mathcal{L}^n\in\mathcal{C}(\mathscr{S}_0)\right)\ge \mathbb{P}_{x_0}\left(\tau_2>n,\mathcal{Z}_0|_E=e_0\right),
$
which deduces that (by Theorem \ref{m-exp decay})
\[
-\lim_{n\to\infty}\frac{1}{n}\log\mathbb{P}_{x_0}\left(\tau_2>n,\mathcal{Z}_0|_E=e_0\right)\ge-\lim_{n\to\infty}\frac{1}{n}\log\mathbb{P}_{x_0}
\left(\mathcal{L}^n\in\mathcal{C}(\mathscr{S}_0)\right)=\alpha_c(\delta).
\]
Noting
$\mathbb{P}_{x_0}\left(\tau_2>n,\mathcal{Z}_0|_E=e_0\right)
\ge\frac{1}{\b}\left(\frac{\delta}{\b-1+\delta}\right)^{n-1},$
we have that as $\delta\to \infty,$
$$-\lim_{n\to\infty}\frac{1}{n}\log\mathbb{P}_{x_0}\left(\tau_2>n,\mathcal{Z}_0|_E=e_0\right)\le \log\frac{\b-1+\delta}{\delta}\to 0.$$
By the squeeze theorem, $\lim\limits_{\delta\to\infty}\alpha_c(\delta)=0$.\end{proof}

\begin{proof}[Proof of Theorem \ref{thm_prop_alpha_c} (c)]
By Proposition \ref{property of alpha_c 2}, we only need to prove
\begin{itemize}
	\item[(i)] The statement
	\begin{center}\emph{There are some $E_0\in\mathscr{S}_0$ and some nonempty set $E'\subseteq E_0$ with $E'\cap \partial E_0=\emptyset$ such that $E'$ is the support of the first marginal measure of some invariant measure on $V_S$.}
	\end{center}
       is  equivalent to the following statement
	\begin{center}\emph{
		For some edge $e\in E$ and some $E_0\in\mathscr{S}_0$, $E_e=\{e'\in E:e'\cap e\neq\emptyset \}\subseteq E_0$.}
	\end{center}

	\item[(ii)] \[
	\inf_{(\mu,E_0\in\mathscr{S}_0):\,\text{\rm supp}(\mu|_E)\subseteq E_0\setminus \partial E_0}\Lambda_1(\mu)=
	\inf_{\stackrel{(\nu,\E_0\in\mathscr{S}_0,\hat{q}):\ \nu\in\mathscr{P}(V), \nu\hat{q}=\nu}{ E({\rm supp}(\nu))\subseteq E_0\setminus \partial E_0 }}\int_V R(\hat{q}\|\hat{p})\ {\rm d}\nu.
	\]
\end{itemize}

\noindent{\bf (i)} If the first statement holds, choose an edge $e\in E'$. Since $E'\cap \partial E_0=\emptyset$, $e\notin \partial E_0$, i.e., there is no edge in $E\setminus E_0$ adjacent to $e$. Therefore, $E_e\subseteq E_0$.

If the second statement holds, set $E'=\{uv\}$. We may see that $E'$ is the support of invariant measure $\mu$ with $\mu(\overrightarrow{uv})=\mu(\overrightarrow{vu})=\frac{1}{2}$. Since there is no edge in $E\setminus E_0$ adjacent to $uv$, we obtain $uv\notin \partial E_0$, which implies $E'\cap \partial E_0=\emptyset$.

\noindent{\bf (ii)} By Lemma \ref{inf attainment},
\begin{align}
&\inf_{(\mu,E_0\in\mathscr{S}_0):\,\text{\rm supp}(\mu|_E)\subseteq E_0\setminus \partial E_0}\Lambda_1(\mu)\nonumber\\
&=
\inf_{(\mu,E_0\in\mathscr{S}_0):\,\text{\rm supp}(\mu|_E)\subseteq E_0\setminus \partial E_0}\left(\inf_{\hat{q}\in\mathscr{T}_G:\nu(z^-)\hat{q}(z^-,z^+)=\mu(z)}\int_{V}R(\hat{q}\|\hat{p})\ {\rm d}\nu\right),
\label{eq-lim-alpha-c-1}
\end{align}
where $\nu=T(\mu)$ for any $\mu\in\mathscr{P}(V_S)$. Since for any invariant $\nu\in\mathscr{P}(V)$ and transition kernel $\hat{q}$ with $\nu\hat{q}=\nu$, $\mu(z):=v(z^-)\hat{q}(z^-,z^+)$ is an invariant measure for transition kernel $q(z,z'):=\hat{q}(z^+,(z')^+)$. We simplify
\eqref{eq-lim-alpha-c-1} as
\[
\inf_{\stackrel{(\nu,\E_0\in\mathscr{S}_0,\hat{q}):\ \nu\in\mathscr{P}(V), \nu\hat{q}=\nu}{ \nu(z^-)\hat{q}(z^-,z^+)>0 \text{ only if } z|_E\in E_0\setminus \partial E_0 }}\int_V R(\hat{q}\|\hat{p})\ {\rm d}\nu.
\]
Next we prove that under the condition $\nu\hat{q}=\nu$,
\begin{equation}\label{4.2-longlrarrow}
 E({\rm supp}(\nu))\subseteq E_0\setminus \partial E_0 \Longleftrightarrow
\nu(z^-)\hat{q}(z^-,z^+)>0 \text{ only if } z|_E\in E_0\setminus \partial E_0.
\end{equation}

Firstly, we verify  that the left hand side (LHS) of (\ref{4.2-longlrarrow}) implies the RHS (\ref{4.2-longlrarrow}).

For $z$ with $\nu(z^-)\hat{q}(z^-,z^+)>0$, then $\nu(z^-)>0$ and $\nu(z^+)>0$ by $\nu \hat{q}=\nu$, which implies $z|_E\in E({\rm supp}(\nu))$. Hence  $z|_E\in E_0\setminus \partial E_0$.

Secondly,  by contradiction, we show  that the RHS of (\ref{4.2-longlrarrow}) implies the LHS (\ref{4.2-longlrarrow}). If the first condition does not hold, choose some $uv\in E({\rm supp}(\nu))$ with $uv\notin E_0\setminus \partial E_0$. Set $z=\overrightarrow{uv}$ for convenience. Note that $\nu(z^-),\nu(z^+)>0$, and that $\nu(z^-)\hat{q}(z^-,z^+)=\nu(z^+)\hat{q}(z^+,z^-)=0$ because of $uv\notin E_0\setminus \partial E_0$ and the second condition. We deduce that $\hat{q}(z^-,z^+)=\hat{q}(z^+,z^-)=0$. By $\nu\hat{q}=\nu$, there exist some $z_1,z_2$ with $z_1^+=z^-, z_2^+=z^+$ such that $\nu(z_i^-)\hat{q}(z_i^-,z_i^+)>0$ for $i=1,2$. Thus by the second condition, $z_1|_E,z_2|_E\in E_0\setminus \partial E_0$. Note that $uv=z|_E$ is adjacent to $z_1|_E,z_2|_E$. This contradicts to $uv\notin E_0\setminus \partial E_0$: for $uv\notin E_0$, $z_1|_E,z_2|_E\in \partial E_0$ if $z_1|_E,z_2|_E\in E_0$; for $uv\in \partial E_0$, there is an edge $u'v'\notin E_0$ adjacent to $uv$, assuming $u'v'\cap uv=z^-=z_1^+$, which implies $z_1|_E\in \partial E_0$ if $z_1|_E\in E_0$.

By (\ref{4.2-longlrarrow}), we finally simplify the RHS of \eqref{eq-lim-alpha-c-1} as
$
\inf_{\stackrel{(\nu,\E_0\in\mathscr{S}_0,\hat{q}):\ \nu\in\mathscr{P}(V), \nu\hat{q}=\nu}{ E({\rm supp}(\nu))\subseteq E_0\setminus \partial E_0 }}\int_V R(\hat{q}\|\hat{p})\ {\rm d}\nu.
$
\end{proof}


\begin{appendix}\label{appendices}



	

  \section{Proofs of Proposition \ref{another expression of alpha_c}} \label{sec-pf-prop-3.1}

\begin{proof}

First, we show (\ref{exprs1 alpha_c}).
 Noting that
	\begin{equation}\label{prop4.2-1}
	  \inf_{\mu\in\mathcal{C}(\mathscr{S}_0)}\Lambda_{\delta,\mathscr{S}_0}(\mu)= \inf_{(\mu_k,r_k,E_k)_k\in\bigcup_{\mu\in\mathcal{C}(\mathscr{S}_0)}\mathscr{A}(\mu,\mathscr{S}_0)}\ \left(\inf_{\qq_k\in\mathscr{T}_S:\,\mu_k\qq_k=\mu_k}\sum_{k=1}^\b r_k \int_{V_S} R(\qq_k\|\pp_{E_k})\ {\rm d}\mu_k\right).	
	\end{equation}

	On the one hand,  for any $E_0\in\mathscr{S}_0$, transition kernel $q$ and invariant measure $\mu$ with ${\rm supp}(\mu|_E)\subseteq E_0$, $\mu \qq=\mu$, we set
		\[
	r_k=1 \text{ if } k=|E_0|, \text{ and }
	0 { \rm otherwise.}
	\]
Due to $E_0\in\mathscr{S}_0$, there exists a $v\in V$ such that $\{x, v\}\in E_0$. Thus, by induction for $k\ge |E_0|$ and $k\le |E_0|$ respectively, we can construct a sequence of $\{E_k\}_{1\le k\le \b}\in\mathscr{E}$ such that $E_{|E_0|}=E_0$. Choosing some $\{\mu_k\}_{1\le k\le \b}$ such that $\mu_{|E_0|}=\mu$, ${\rm supp}(\mu_k|_E)\subseteq E_k$ for $1\le k\le \b$, we may see that \[(\mu_k,r_k,E_k)_k\in\cup_{\mu\in\mathcal{C}(\mathscr{S}_0)}\mathscr{A}(\mu,\mathscr{S}_0).\]
	Therefore, $\inf_{\mu\in\mathcal{C}(\mathscr{S}_0)}\Lambda_{\delta,\mathscr{S}_0}(\mu)\le \sum_{k=1}^\b r_k \int_{V_S} R(\qq_k\|\pp_{E_k})\ {\rm d}\mu_k=\int_{V_S} R(\qq\|\pp_{E_0})\ {\rm d}\mu$. By the arbitrariness of $(\mu, E_0, \qq)$, we have
	\begin{equation}\label{prop4.2-3}
	\inf_{\mu\in\mathcal{C}(\mathscr{S}_0)}\Lambda_{\delta,\mathscr{S}_0}(\mu)\le \inf_{(\mu,E_0\in\mathscr{S}_0,\qq):\,\text{\rm supp}(\mu|_E)\subseteq E_0,\mu \qq=\mu}\int_{V_S} R(\qq\|\pp_{E_0})\ {\rm d}\mu.
	\end{equation}

	On the other hand,  using (\ref{prop4.2-1}), we have
	\begin{align}
	&\inf_{\mu\in\mathcal{C}(\mathscr{S}_0)}\Lambda_{\delta,\mathscr{S}_0}(\mu)\nonumber\\
	&\ge \inf_{(\mu_k,r_k,E_k)_k\in\bigcup_{\mu\in\mathcal{C}(\mathscr{S}_0)}\mathscr{A}(\mu,\mathscr{S}_0)}\ \left[\inf_{\qq_k\in\mathscr{T}_S:\,\mu_k\qq_k=\mu_k}\left(\min_{1\le k\le \b: E_k\in\mathscr{S}_0} \int_{V_S} R(\qq_k\|\pp_{E_k})\ {\rm d}\mu_k\right)\right].\label{prop4.2-2}
	\end{align}	
	Noting that if $(\mu_k,r_k,E_k)_k\in\cup_{\mu\in\mathcal{C}(\mathscr{S}_0)}\mathscr{A}(\mu,\mathscr{S}_0)$ with $\mu_k\qq_k=\mu_k$, then  ${\rm supp}(\mu_k|_E)\subseteq E_k$. Therefore, we deduce that for all $k$ with $E_k\in \mathscr{S}_0$,
	\[
	\int_{V_S} R(\qq_k\|\pp_{E_k})\ {\rm d}\mu_k\ge \inf_{(\mu,E_0,\qq):\,E_0\in\mathscr{S}_0,\,\text{\rm supp}(\mu|_E)\subseteq E_0,\,\mu \qq=\mu}\int_{V_S} R(\qq\|\pp_{E_0})\ {\rm d}\mu.
	\]
	Combining with (\ref{prop4.2-2}), we obtain
	\[
	\inf_{\mu\in\mathcal{C}(\mathscr{S}_0)}\Lambda_{\delta,\mathscr{S}_0}(\mu)\ge \inf_{(\mu,E_0\in\mathscr{S}_0,q):\,\text{\rm supp}(\mu|_E)\subseteq E_0,\mu \qq=\mu}\int_{V_S} R(\qq\|\pp_{E_0})\ {\rm d}\mu.
	\]
	(\ref{prop4.2-3}) and (\ref{prop4.2-2}) imply the expression (\ref{exprs1 alpha_c}) of  $\inf_{\mu\in\mathcal{C}(\mathscr{S}_0)}\Lambda_{\delta,\mathscr{S}_0}(\mu)$.\end{proof}

\section{Proof of Lemma \ref{order of Markov process}}\label{apdx-mrkproc}

\begin{proof}
	Without loss of generality, we may consider another Markov process $\{\xi'_n\}_{n\ge0}$ by modifying the transition probability from states in $A$ such that $A$ is an absorbing subset of $\{\xi'_n\}_{n\ge0}$. Then $\P(\tau> n)=\P(\tau'> n)$, where $\tau':=\inf\{n: \xi'_n\in A \}$. Meanwhile, we can assume that each state in $A^c$ is accessible for $\{\xi'_n\}_{n\ge0}$ starting at any state in $A^c$. Otherwise, for fixed starting state, we may just consider those accessible states in $A^c$.
	
	By the proof of part (e) of
	\cite[p.\,72, Theorem 3.1.1]{DZ1998} (the Perron-Frobenius theorem), choosing irreducible $B=(p_{ij})_{A^c\times A^c}$, we obtain that
	\[
	\sum_{j\in A^c} B^n(i,j)\phi_j=\Theta\left(p_0^n\right),\ i\in A^c,
	\]
	where $\phi$ is some measure on $A^c$ with $\phi_j>0$ for each $j\in A^c$, and $p_0\in (0,1)$ is the maximal eigenvalue of $B$,
	and $B^n(i,j)$ is the $(i,j)$-entry of $B^n.$ This implies
	$\P(\tau'>n\vert\xi'_0=i)=\Theta\left(p_0^n\right),\ i\in A^c,$
	which deduces that $\P(\tau>n)=\sum_{i\in A^c}\P(\xi'_0=i)\cdot\P(\tau'>n\vert\xi'_0=i)=\Theta\left(p_0^n\right)$. \end{proof}

\section{Proofs of Proposition \ref{another expression of alpha_c 1}} \label{sec-pf-prop-4.1}

\begin{proof}
	
\noindent\emph{(a).}	Now, it is time to show the expression (\ref{exprs2 alpha_c}) of  $\inf_{\mu\in\mathcal{C}(\mathscr{S}_0)}\Lambda_{\delta,\mathscr{S}_0}(\mu)$. 	By (\ref{exprs1 alpha_c}),
 note that
\begin{align}
\int_{V_S} R(\qq\|\pp_{E_0})\ {\rm d}\mu &=\int_{V_S} R(\qq\|\pp)\ {\rm d}\mu +\int_{V_S} \log{\frac{d(z)-k(z)+k(z)\delta}{d(z)\delta}}\ {\rm d}\mu\nonumber\\
&=\int_{V_S} R(\qq\|\pp)\ {\rm d}\mu +\int_{\{\overrightarrow{uv}: uv\in \partial E_{0}\}} \log{\frac{d(z)-k(z)+k(z)\delta}{d(z)\delta}}\ {\rm d}\mu.\label{(4.21)}
\end{align}
Due to  \cite[Theorem 8.4.3]{DE1997}, $\Lambda_1(\mu)=\inf_{\qq\in\mathscr{T}_S:\,\mu \qq=\mu}\int_{V_S} R(\qq\|\pp)\ {\rm d}\mu$.

Hence we obtain (\ref{exprs2 alpha_c}) by taking $\inf_{(\mu,E_0\in\mathscr{S}_0,\qq):\,\text{\rm supp}(\mu|_E)\subseteq E_0,\mu \qq=\mu}\{\cdot\}$ on both sides of \eqref{(4.21)}.\\

\noindent\emph{(b).}  Set $g(\mu):=\Lambda_{1}(\mu)+\int_{\{\overrightarrow{uv}: uv\in \partial E_{0}\}}\log{\frac{d(z)-k(z)+k(z)\delta}{d(z)\delta}}\ {\rm d}\mu,\ \mu\in\mathscr{P}(V_S).$ By \eqref{exprs2 alpha_c},
\[
\inf_{\mu\in\mathcal{C}(\mathscr{S}_0)}\Lambda_{\delta,\mathscr{S}_0}(\mu)=\inf_{(\mu,E_0):E_0\in\mathscr{S}_0,\,\text{\rm supp}(\mu|_E)\subseteq E_0}g(\mu).
\]
We prove that the infimum of $g(\mu)$ is attained at some $(\mu',E_0)$ with ${\rm supp}(\mu|_E)=E_0$.

The  strategy of proof is to show  for fixed $E_0\in\mathscr{S}_0$ and any $\mu'$ with ${\rm supp}(\mu'|_E)\subsetneq E_0$, there exists $\tilde{\mu}$ with ${\rm supp}(\mu'|_E)\subsetneq{\rm supp}{(\tilde{\mu}|_E})$ such that
\[
g(\tilde{\mu})<g(\mu').
\]
Thus, by recursion,  the finiteness of $E_0$ implies the result that we need. Specifically, we construct a family of probability measure $\{\mu^\varepsilon, 0\le \varepsilon< 1\}\in \mathscr{P}(V_S)$ satisfying:
\begin{itemize}
  \item $\mu^0=\mu'$;
  \item ${\rm supp}(\mu'|_E)\subsetneq{\rm supp}{(\mu^\varepsilon|_E})$, for $0<\varepsilon<1$;
  \item $\lim\limits_{\varepsilon\to 0}\mu^\varepsilon(z)= \mu^0(z), z\in V_S$.
\end{itemize}
$\{\mu^\varepsilon\}$ can be regarded as some perturbation of  $\mu'$. Then, by analyzing the positive and negative properties of the derivative of $g(\mu^\varepsilon)$ with respect to $\varepsilon$ near $\varepsilon=0$, we get  there exists a small $\varepsilon>0$ such that $g(\mu^\varepsilon)<g(\mu')$. For this purpose, we need to a precise expression of $\mu^\varepsilon$ as follows:

Note that \[g(\mu')=\inf_{\qq\in\mathscr{T}_S:\,\mu'\qq=\mu'}\int_{V_S}R(\qq\|\pp)\ {\rm d}\mu'+\int_{\{\overrightarrow{uv}: uv\in \partial E_{0}\}} \log{\frac{d(z)-k(z)+k(z)\delta}{d(z)\delta}}\ {\rm d}\mu',\]
and assume that the infimum of $\inf_{\{\qq\in\mathscr{T}_S:\,\mu'\qq=\mu'\}}\int_{V_S}R(\qq\|\pp)\ {\rm d}\mu'$ is attained at some $\qq'\in\mathscr{T}_S$.  Denote by $V_1',\dots, V_r'$ the absorbing sets of the transition kernel $\qq'$. Set $c_k:=\mu'(V_k')>0$ for $k=1,\cdots, r$, where $\sum_{k=1}^rc_k=1$.   Because  ${\rm supp}(\mu'|_E)\subsetneq E_0$, there must be  an $uv\in E_0$ with $\mu'|_E(uv)=0$ such that $\overrightarrow{uv},\overrightarrow{vu}$ do not belong to all $V'_k,\, k=1,\cdots, r$.

Choose an absorbing set, assumed to be $V_1'$, such that there exists some $u_1v_1\in E_0$ with $\mu'|_E(u_1v_1)=0$ and some $z_0=\overrightarrow{u_0v_0}\in V_1'$ with $z_0\to \overrightarrow{u_1v_1}$ or $z_0\to\overrightarrow{v_1u_1}$ by an edge in $S$. Without loss of generality, assume $z_0\to \overrightarrow{u_1v_1}$.
Let $k'(x)=|\{ y: \qq'(x,y)>0 \}|$. We then construct a  transition probability $\qq^\varepsilon\in\mathscr{T}_S$ for $\varepsilon \in [0,1)$. For convenience, sometimes, we also write $\qq(x,y)$ as $\qq_{xy}$ for some $\qq\in\mathscr{T}_S$.

\begin{itemize}
\item[(a)] If $z_0^*:=\overrightarrow{v_0u_0}\notin V_1'$, set
\[
\qq^\varepsilon_{xy}=\left\{
\begin{array}{ll}
\qq'_{xy}-\varepsilon/k'(x), & x=z_0,\, \qq'_{xy}>0,\\
\varepsilon, & x=z_0,\, y=z_0^*,\\
1, & x=z_0^*,\, y=z_0,\\
\qq'_{xy}, & x\neq z_0,\, z_0^*.
\end{array}
\right.
\]
\item[(b)] If $z_0^*\in V_1'$, set
\[
\qq^\varepsilon_{xy}=\left\{
\begin{array}{ll}
\qq'_{xy}-\varepsilon/k'(x), & x=z_0,\, \qq'_{xy}>0,\\
\varepsilon, & x=z_0,\, y= \overrightarrow{u_1v_1},\\
1, & x=\overrightarrow{u_1v_1},y=\overrightarrow{v_1u_1},\, \text{ or }x=\overrightarrow{v_1u_1},\, y=z_0^*,\\
\qq'_{xy}, & x\neq z_0,\, \overrightarrow{u_1v_1},\overrightarrow{v_1u_1}.
\end{array}
\right.
\]
\end{itemize}

\begin{figure}
		\subfigure[]{\begin{minipage}[t]{0.5\textwidth}
				\centering{\includegraphics[scale=1]{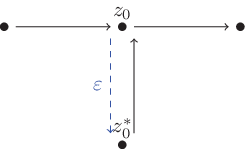}}
		\end{minipage}}
		\subfigure[]{\begin{minipage}[t]{0.5\textwidth}
				\centering{\includegraphics[scale=1]{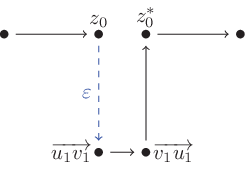}}
		\end{minipage}}
		\caption{\label{lemma_4_9_graph} \small{(a) and (b) stand for the cases (a) and (b) of $\qq^\varepsilon$ respectively.}}
\end{figure}
\noindent (See Figure \ref{lemma_4_9_graph} for an intuition of $\qq^\varepsilon$.)  Let
\[
\ V_1'':=\left\{
\begin{aligned}
&V_1'\cup \{z_0^*\}, &z_0^*\notin V_1',\\
&V_1'\cup \{\overrightarrow{u_1v_1},\overrightarrow{v_1u_1}\}, &z_0^*\in V_1'.
\end{aligned}
\right.
\]

Now, we construct a new probability measure $\mu^\varepsilon\in\mathscr{P}(V_S)$ such that
\begin{equation}\label{4-mu}
  \left\{ \begin{aligned}
    \mu^\varepsilon \qq^\varepsilon=&\,\mu^\varepsilon,  \\
     \mu^\varepsilon(V_1'')=&\, c_1, \\
     \cdots &   \\
     \mu^\varepsilon(V_k')=&\,c_k\ {\rm for}\ k=2,\dots,r.
  \end{aligned}
  \right.
\end{equation}
That is to say, $\mu^\varepsilon$ is an invariant distribution  of $\qq^\varepsilon$ and supports on $V_1''$ with probability $c_1$ and on $V_k^\varepsilon$ with probability $c_r$,  $r=2,\cdots,r$.  By the communicating classes and the property of invariant distribution of Markov chains with finite state spaces,
the above equation (\ref{4-mu}) has a unique solution. In particular,
\begin{itemize}
  \item  if $z\in V_S\setminus V_1''$, $\qq^\varepsilon(z, \cdot)=
  \qq^0(z,\cdot)$, then $\mu^\varepsilon(z)=\mu'(z)$ for any $\varepsilon\ge 0$; 
  \item if $\varepsilon>0$,  then $\mu^\varepsilon(V_1''\setminus V_1')>0$;
  \item if $\varepsilon=0$,  then $\mu^0(V_1''\setminus V_1')=0$ and if $z\in V_1'$, then
$\mu'(z)=\mu^0(z)$.
\end{itemize}
These properties imply
\begin{equation}\label{4-mu-1}
  \mu'=\mu^0\ \  \text{and}\ \  {\rm supp}(\mu^0|_E)\subsetneq{\rm supp}{(\mu^\varepsilon|_E}),\  0<\varepsilon<1.
\end{equation}
Note that for any $k=1,\cdots, r$,  if $x \in V_k'$, $\qq'(x,\cdot)=\qq^0(x,\cdot)$, we have
\begin{equation}\label{4-mu-2}
  \int_{V_S}R(\qq'\|\pp)\ {\rm d}\mu'=\int_{V_S}R(\qq^0\|\pp)\ {\rm d}\mu^0.
\end{equation}
Besides, we only need to consider those vertices in $V_0'$ when proving $g(\mu')>g(\mu^\varepsilon)$.

Moreover,  we have the following three claims:

\begin{itemize}
    \item[{\bf (a)}] $\lim\limits_{\varepsilon\to 0}\mu^\varepsilon(z)= \mu^0(z), z\in V_1''$.
     \item[{\bf (b)}] $\frac{{\rm d} \mu^\varepsilon(z)}{{\rm d}\varepsilon}$ is uniformly bounded for all $z\in V_1''$ and $\varepsilon\in[0,\frac{1}{2}\max_{\{y:\qq_{z_0y}'>0\}}\{\qq_{z_0y}'\cdot k'(z_0)\}]$.
    \item[{\bf (c)}]
    $
    \frac{{\rm d}\int_{V_S} R(\qq^\varepsilon\|\pp)\ {\rm d}\mu^\varepsilon}{{\rm d} \varepsilon} \to-\infty\ \mbox{for}\ \varepsilon\to 0.
    $
\end{itemize}

Set
\[
\tilde{g}(\qq^\varepsilon,\mu^\varepsilon):=\int_{V_S}R(\qq^\varepsilon\|\pp)\ {\rm d}\mu^\varepsilon+\int_{\{\overrightarrow{uv}: uv\in \partial E_{0}\}} \log{\frac{d(z)-k(z)+k(z)\delta}{d(z)\delta}}\ {\rm d}\mu^\varepsilon.
\]
Assume temporarily the above three claims holds.   Claims {\bf (a)}, {\bf (b)} and {\bf (c)} imply that  $\tilde{g}(\qq^\varepsilon,\mu^\varepsilon)$ is not only continuous, but also strictly decreasing in small $\varepsilon$. Therefore, due to  (\ref{4-mu-1}) and  (\ref{4-mu-2}), we obtain that there is an $\varepsilon>0$ such that
\[
g(\mu')=g(\mu^0)=\tilde{g}(\qq^0,\mu^0)> \tilde{g}(\qq^\varepsilon,\mu^\varepsilon) \ge g(\mu^\varepsilon).
\]
This illustrates that the infimum of $g$ must be attained at some $\mu\in\mathscr{P}(V_S)$ with $\mu|_E(\partial E_0)>0$.

Now, let us prove the three claims.\ 

Due to the irreducibility of $\qq^\varepsilon$ on $V_1''$ for any $\varepsilon\in [0,1)$, 
\begin{equation}\label{4-mu-5}
\left\{
 \begin{aligned}
   \sum_{x\in V_1''}\mu^\varepsilon(x)\qq^\varepsilon_{xy}=&\, \mu^\varepsilon(y),\ y\in V_1'',   \\
   \sum_{x\in V_1''}\mu^\varepsilon(x)=&\, c_1.
 \end{aligned}\right.
\end{equation}
has a unique solution which is non-negative. That is to say, for any $\varepsilon\in [0,1)$, $\{\mu^\varepsilon(z)\}_{z\in V_1''}$, which is the solution to Eq. (\ref{4-mu}) restricted on $V_1''$, is uniquely determined by $\{\qq^\varepsilon_{xy}\}_{x,y\in V_1''}$.

In order to verify Claim {\bf (a)}, {\bf (b)} and {\bf(c)}, we have to give the explicit solution to Eq. (\ref{4-mu-5}) through Cr\'amer's rule.  Set matrix $Q^\varepsilon=(\tilde{\qq}_{xy}^\varepsilon)_{x,y\in V_1''}$ with entries
\[
\tilde{\qq}_{xy}^\varepsilon=\left\{\begin{array}{ll}
\qq^\varepsilon_{xy},\ &x\neq y,\,y\neq z_0,\\
-1,\ &x=y\neq z_0,\\
1,\ &y=z_0,
\end{array}
\right.
\]
and $Q_x^\varepsilon$ is the matrix induced from $Q^\varepsilon$ by replacing $\tilde{\qq}_{xy}^\varepsilon$ by $0$ for all $y\neq z_0$ and replacing $\tilde{\qq}_{xz_0}^\varepsilon$ by $c_1$.  Thus,
\[
\mu^\varepsilon(x)=\frac{\text{det}(Q_x^\varepsilon)}{\text{det}(Q^\varepsilon)},\ x\in V_1'',
\]
where $\text{det}(\cdot)$ denotes the determinant of a matrix.
Furthermore,
$$
\mu^\varepsilon(x)=
\left\{ \begin{aligned}
  \frac{c_1\,b}{\sum_{y\in V_0'\setminus\{z_0\}}a_y(\varepsilon)+b}, &\ \ \  x=z_0,\\
\frac{c_1\,a_x(\varepsilon)}{\sum_{y\in V_0'\setminus\{z_0\}}a_y(\varepsilon)+b}, &\ \ \  x\neq z_0.
  \end{aligned}
   \right.
$$
Here $a_x(\varepsilon)$ is the algebraic cofactor of  the entry $\tilde{\qq}_{xz_0}^\varepsilon$ of $Q^\varepsilon$, and $b$ is the algebraic cofactor of the entry $\tilde{\qq}_{z_0z_0}^\varepsilon$ of  $Q^\varepsilon$.  For each $x$,  $a_x(\varepsilon)$ is a linear function of $\varepsilon$;  $b$ is independent of $\varepsilon$.

Moreover,
\[
\sum_{y\in V_0'\setminus\{z_0\}}a_y(0)+b=\text{det}(Q^0)\neq 0
\]
 because of uniqueness and existence of the solution to the following equation:
 $$
  \left\{
 \begin{aligned}
   \sum_{x\in V_1''}\mu^0(x)\qq^0_{xy}=&\, \mu^0(y),\ y\in V_1'',   \\
   \sum_{x\in V_1''}\mu^0(x)=&\, c_1.
 \end{aligned}\right.
 $$

 Hence we demonstrate that $\mu^\varepsilon$ is not only continuous, but also differentiable in $\varepsilon$ with a uniformly bounded derivative for all $z\in V_0'$ and $\varepsilon\in[0,\frac{1}{2}\max_{y\leftarrow z_0}\qq_{z_0y}'\cdot k'(z_0)]$.
This verifies  Claims {\bf (a)} and {\bf (b)}.

Note that
\begin{align*}
\int_{V_S} R(\qq^\varepsilon\|\pp)\ {\rm d}\mu^\varepsilon &=\sum_{x\in V_S}\mu^\varepsilon(x)\sum_{y\in V_S}\log{\frac{\qq^\varepsilon_{xy}}{\pp_{xy}}}\qq^\varepsilon_{xy}\\
&=\sum_{x\in \{\overrightarrow{uv}: uv\in E_{0}\}}\mu^\varepsilon(x)\sum_{y\in \{\overrightarrow{uv}: uv\in E_{0}\}}\left(\qq^\varepsilon_{xy}\log{\qq^\varepsilon_{xy}}-\qq^\varepsilon_{xy}\log{\pp_{xy}}\right),
\end{align*}
and the following facts:
\begin{itemize}
\item $0\le\sum_{y\in V_S}\log{\frac{\qq^\varepsilon_{xy}}{\pp_{xy}}}\qq^\varepsilon_{xy}\le \sup_{x\in V_S}\sum_{y\in V_S, y\leftarrow x}\log\frac{1}{\pp_{xy}}<\infty$.
\item For $x\neq z_0$, $\frac{{\rm d} \qq^\varepsilon_{x\cdot}}{{\rm d} \varepsilon}=0$, and hence
$
\frac{{\rm d}(\qq^\varepsilon_{xy}\log{\qq^\varepsilon_{xy}}-\qq^\varepsilon_{xy}\log{\pp_{xy}})}{{\rm d}\varepsilon}=0.
$
\item For $x=z_0$, $y\in V_1'$ with $z_0\to y$,
$$
\frac{{\rm d} (\qq^\varepsilon_{xy}\log{\qq^\varepsilon_{xy}})}{{\rm d}\varepsilon}=-\frac{1}{k'(z_0)}\cdot\left(\log{\left(\qq'_{xy}-\varepsilon/k'(z_0)\right)}+1\right),
\text{ and } \frac{{\rm d}(\qq^\varepsilon_{xy}\log{\pp^\varepsilon_{xy}})}{{\rm d}\varepsilon}=-\frac{1}{k'(z_0)}\log{\pp_{xy}}.$$
Thus,
$
\lim_{\varepsilon\to 0}\frac{{\rm d}(\qq^\varepsilon_{xy}\log{\qq^\varepsilon_{xy}}-\qq^\varepsilon_{xy}\log{\pp_{xy}})}{{\rm d}\varepsilon}
\ \mbox{exists}.
$
\item For $x=z_0$, $y\in V_1''\setminus V_1'$ with $z_0\to y$,
$$
\frac{{\rm d} (\qq^\varepsilon_{xy}\log{\qq^\varepsilon_{xy}})}{{\rm d}\varepsilon}
=1+\log{\varepsilon}\to -\infty\ \mbox{as}\ \varepsilon\to 0.
$$
Since $\frac{{\rm d}(\qq^\varepsilon_{xy}\log{\pp^\varepsilon_{xy}})}{{\rm d}\varepsilon}=\log{\pp_{xy}}$, we have that
$
\frac{{\rm d}(\qq^\varepsilon_{xy}\log{\qq^\varepsilon_{xy}}-\qq^\varepsilon_{xy}\log{p_{xy}})}{{\rm d}\varepsilon} \to -\infty\ \mbox{as}\ \varepsilon\to 0.
$
\end{itemize}
Combining $\mu^0(z_0)>0$ with  Claim {\bf (b)}, we deduce Claim {\bf (c)} immediately. 	
\end{proof}


\section{Proofs of Proposition \ref{property of alpha_c 2}} \label{sec-pf-prop-4.2}

Before our proof of Proposition \ref{property of alpha_c 2}, for any $\mu\in\mathscr{P}(V_S)$ and some $E_0\in\mathscr{S}_0$, we show two properties:
\begin{align}
&\partial (\{\overrightarrow{uv}: uv\in E_{0}\})\subseteq\{\overrightarrow{uv}: uv\in \partial E_{0}\},\label{prop-partial-1}\\
&\mu(\{\overrightarrow{uv}: uv\in \partial E_{0}\})>0 \text{ implies } \mu(\partial (\{\overrightarrow{uv}: uv\in E_{0}\}))>0.\label{prop-partial-2}
\end{align}

Firstly, we verify property \eqref{prop-partial-1}. For $z=\overrightarrow{u_0v_0}\in\partial (\{\overrightarrow{uv}: uv\in E_{0}\})$, there is some $\tilde{z}\notin \{\overrightarrow{uv}: uv\in E_{0}\}$ such that $z\to \tilde{z}$. Note that $\tilde{z}|_E\notin E_0$, and that $\tilde{z}|_E$ is adjacent to $z|_E=u_0v_0$. Therefore, $u_0v_0\in\partial E_0$, which implies $z\in\{\overrightarrow{uv}: uv\in \partial E_{0}\}$. Since we select $z$ arbitrarily, we obtain \eqref{prop-partial-1}.

Secondly, we verify property \eqref{prop-partial-2}. We claim that for any $z=\overrightarrow{u_0v_0}\in \{\overrightarrow{uv}: uv\in \partial E_{0}\}$,
\[z\in \partial (\{\overrightarrow{uv}: uv\in E_{0}\})\mbox{, or } z^*:=\overrightarrow{v_0u_0}\in\partial (\{\overrightarrow{uv}: uv\in E_{0}\}).\]
This is because $u_0v_0\in \partial E_0$, i.e., there is some $u'v'\notin E_0$ adjacent to $u_0v_0$. Note that $\overrightarrow{u'v'},\overrightarrow{v'u'}\notin \{\overrightarrow{uv}: uv\in E_{0}\}$, and that vertex $u_0v_0\cap u'v'$ belongs to $\{z^+,z^-\}$. If $u_0v_0\cap u'v'=z^+$, then $z\to\overrightarrow{u'v'}$ or $z\to\overrightarrow{v'u'}$, which deduces $z\in \partial (\{\overrightarrow{uv}: uv\in E_{0}\})$. Otherwise, $z^*\to\overrightarrow{u'v'}$ or $z^*\to\overrightarrow{v'u'}$, which deduces $z^*\in \partial (\{\overrightarrow{uv}: uv\in E_{0}\})$.

Here we prove $\mu(\partial( \{\overrightarrow{uv}: uv\in E_{0}\}))>0$ by contradiction. If $\mu(\partial( \{\overrightarrow{uv}: uv\in E_{0}\}))=0$, then there is some $z=\overrightarrow{u_0v_0}\in \{\overrightarrow{uv}: uv\in \partial E_{0}\}$ with $\mu(z)>0$ such that $z^*=\overrightarrow{v_0u_0}\in \partial( \{\overrightarrow{uv}: uv\in E_{0}\})$, i.e., there is some $\tilde{z}\notin \{\overrightarrow{uv}: uv\in E_{0}\} $ with $z^*\to \tilde{z}$. Since $\mu$ is an invariant measure, there exists some $z'\to z$ such that $\mu(z')>0$. Noting that $(z')^+=z^-=(z^*)^+=\tilde{z}^-$, we deduce $z'\to \tilde{z}$, i.e., $z'\in \partial( \{\overrightarrow{uv}: uv\in E_{0}\})$, which contradicts to $\mu(\partial( \{\overrightarrow{uv}: uv\in E_{0}\}))=0$.

With properties \eqref{prop-partial-1} and \eqref{prop-partial-2}, we start the proof of Proposition \ref{property of alpha_c 2}.


\begin{proof}[Proof of Proposition \ref{property of alpha_c 2}]
Suppose firstly that for all nonempty sets $E',E_0$ with $E'\subseteq E_0\in\mathscr{S}_0$ and $E'\cap \partial E_0\neq \emptyset$, there is some invariant measure $\nu$ with $\text{supp}(\nu\vert_1)=E'$. In this case, for any invariant measure $\mu$ and $E_0\in\mathscr{S}_0$ with $\text{supp}(\mu|_E)\subseteq E_0$, we have $\mu(\{\overrightarrow{uv}: uv\in \partial E_{0}\})>0$. This implies $\mu(\partial( \{\overrightarrow{uv}: uv\in E_{0}\}))>0$ by \eqref{prop-partial-2}.

\noindent {\bf (a)\ \boldmath $\lim\limits_{\delta\to 0}\alpha_c(\delta)=+\infty$ }
Here we use the expression $\alpha_c(\delta)=\inf_{\mu\in\mathcal{C}(\mathscr{S}_0)}\Lambda_{\delta,\mathscr{S}_0}(\mu)$. Consider $\delta<1$. For each $\mu\in\mathcal{C}(\mathscr{S}_0)$, we assume $\Lambda_{\delta,\mathscr{S}_0}(\mu)$ is attained at $(\mu_{k},r_k,E_k)_{1\le k\le \b}$. Note that
\begin{align*}	
R\left(\qq_{k}(x,\cdot)\|\pp(x,\cdot)\right)-R\left(\qq_{k}(x,\cdot)\|\pp_{E_k}(x,\cdot)\right)
=\int_{V_S} \log{\frac{\pp_{E_k}(x,y)}{\pp(x,y)}}\qq_{k}(x,{\rm d}y),	
\end{align*}
where $\log{\frac{\pp_{E_k}(x,y)}{\pp(x,y)}}\le \log{\frac{d\delta}{(d-1)\delta+1}}$ for $\delta<1$ and $x\in\partial (\{\overrightarrow{uv}:uv\in E_k\})$. Meanwhile, for each measurable set $A$ and all $k\le N:=\sup\{k:\ E_k\in\mathscr{S}_0\}$, then
\begin{align*}
\sum_{k=1}^{N}r_k\mu_{k}(A)&=\mu(A),\\
\mu(\{\overrightarrow{uv}: uv\in \partial E_{N}\})&\ge\min_{z\in\text{supp}(\mu)}\mu(z)>0,\\
\mu_k(\partial(\{\overrightarrow{uv}: uv\in E_{k}\}))&\ge\mu_k(\partial(\{\overrightarrow{uv}: uv\in \partial E_{N}\})).
\end{align*}
The third inequality above follows from that \[\partial(\{\overrightarrow{uv}: uv\in E_{N}\})\subseteq \partial(\{\overrightarrow{uv}: uv\in E_{k}\})\cup\left(\{\overrightarrow{uv}: uv\in E_{k}\}\right)^c,\]
and ${\rm{supp}}(\mu_k)=\{\overrightarrow{uv}: uv\in E_{k}\}.$ Thus
\begin{align*}
\Lambda_1(\mu)-\Lambda_{\delta,\mathscr{S}_0}(\mu)&\le\sum_{k=1}^{N}r_k\mu_{k}(\partial(\{\overrightarrow{uv}:uv\in E_k  \}))\log{\frac{\b\delta}{(\b-1)\delta+1}}\\
&\le\mu(\partial (\{\overrightarrow{uv}: uv\in E_{N}\}))\log{\frac{\b\delta}{(\b-1)\delta+1}}\\
&\le \min_{z\in\text{supp}(\mu)}\mu(z)\log{\frac{\b\delta}{(\b-1)\delta+1}}.
\end{align*}
Since $\Lambda_1(\mu)\ge 0$ and $\log{\frac{\b\delta}{(\b-1)\delta+1}}\to-\infty\ (\delta\to 0),$ we get $\lim\limits_{\delta\to 0}\Lambda_{\delta,\mathscr{S}_0}(\mu)=+\infty.$

By the monotonicity of $\alpha_c(\delta)$, $\lim\limits_{\delta\to0}\alpha_c(\delta)$ exists. Assume $\lim\limits_{\delta\to0}\alpha_c(\delta)<+\infty.$ Then there are some sequence $(\delta_n,\mu_n)_{n\geq 1}\subset (0,1)\times\mathscr{P}(V_S)$ and some $M>0$ with $\Lambda_{\delta_n,\mathscr{S}_0}(\mu_n)\le M$ and $\delta_n\downarrow 0.$ Since $\mathscr{P}(V_S)$ is compact, and $\mathcal{C}(\mathscr{S}_0)$ is closed, we know that $\mathcal{C}(\mathscr{S}_0)$ is also compact, and that there are a subsequence $(\mu_{n_k})_{k\geq 1}$ of $(\mu_n)_{n\geq 1}$ and a $\mu\in\mathcal{C}(\mathscr{S}_0)$ with $\mu_{n_k}\Rightarrow \mu$.
For all $\varepsilon>0$ and $j\geq 1$, by the lower semicontinuity of $\Lambda_{\delta_{n_j},\mathscr{S}_0}$, there exists some $K_j>j$ such that for all $i>K_j$, $\Lambda_{\delta_{n_j},\mathscr{S}_0}(\mu_{n_i})\ge\Lambda_{\delta_{n_j},\mathscr{S}_0}(\mu)-\varepsilon$. Hence, by the decreasing of $\Lambda_{\delta,\mathscr{S}_0}$ in $\delta$,
$$
\Lambda_{\delta_{n_i},\mathscr{S}_0}(\mu_{n_i})\ge\Lambda_{\delta_{n_j},\mathscr{S}_0}(\mu_{n_i})
\ge\Lambda_{\delta_{n_j},\mathscr{S}_0}(\mu)-\varepsilon.
$$
By taking $\liminf\limits_{i\to\infty}$ and then letting $\epsilon\to 0$, we obtain that \[\Lambda_{\delta_{n_j},\mathscr{S}_0}(\mu)\le\liminf_{i\to\infty}\Lambda_{\delta_{n_i},\mathscr{S}_0}(\mu_{\delta_{n_i}})\le M,\ j\geq 1,\]
which contradicts to $\lim\limits_{\delta\to0}\Lambda_{\delta,\mathscr{S}_0}(\mu)=+\infty$ due to $\delta_{n_j}\downarrow 0\ (j\uparrow\infty).$

Now we turn to the opposite case.

\noindent {\bf (b)\ \boldmath $\lim\limits_{\delta\to 0}\alpha_c(\delta)<+\infty$ }
Choose $E_0\in \mathscr{S}_0$ and any $\mu\in\mathscr{P}(V_S)$ such that
$${\rm supp}(\mu\vert_1)=E',\ E'\subseteq E_0,\ E'\cap \partial E_0=\emptyset.$$
Then for any transition probability $\qq$ with $\mu \qq=\mu$,
\begin{equation}
\int_{V_S} R(\qq\|\pp_{E_0})\ {\rm d}\mu = \int_{V_S} R(\qq\|\pp)\ {\rm d}\mu. \label{(4.15)}
\end{equation}
This implies $\alpha_c(\delta)\le \inf_{\qq:\,\mu \qq=\mu}\int_{V_S} R(\qq\|\pp)\ {\rm d}\mu=\Lambda_1(\mu)<\infty$. Therefore,
\[\alpha_c(\delta)\le b:=\inf_{(\mu,E_0\in\mathscr{S}_0):\,\text{\rm supp}(\mu|_E)\subseteq E_0\setminus \partial E_0}\Lambda_1(\mu)<\infty.\]
By the expression (\ref{exprs1 alpha_c}) in Proposition \ref{another expression of alpha_c}, the infimum is attained at some $\mu_1$ with $\mu_1|_E(\partial E_0)>0$.
In this case, (\ref{proof of monotonicity of alpha_c}) still holds, namely $\alpha_c(\delta)$ is strictly decreasing in $\delta.$
Clearly $\lim_{\delta\to0}\alpha_c(\delta)\le b$. For any integer $n\geq 2$, assume the infimum of $\alpha_c(1/n)$ is attained at $\mu_n$ and $E_n\in\mathscr{S}_0$. Set $f_n(z)=\log{\frac{d(z)-k(z)+k(z)/n}{d(z)/n}}$. Note that $f_n(z)=0$ if $z\notin \partial(\{\overrightarrow{uv}: uv\in E_{0}\})$, and that $\partial(\{\overrightarrow{uv}: uv\in E_{0}\})\subseteq \{\overrightarrow{uv}: uv\in \partial E_{0}\}$ by \eqref{prop-partial-1}.
By (\ref{exprs2 alpha_c}), \[\alpha_c(1/n)=\Lambda_1(\mu_n)+\int_{\{\overrightarrow{uv}: uv\in \partial E_{n}\}}f_n\ {\rm d}\mu_n=\Lambda_1(\mu_n)+\int_{\partial( \{\overrightarrow{uv}: uv\in E_{n}\})}f_n\ {\rm d}\mu_n,\]
where $\mu_n|_E(\partial E_n)>0$, and hence $\mu_n(\partial (\{\overrightarrow{uv}: uv\in E_{n}\}))>0$. Since $S$ is finite, there is an strictly increasing subsequence $\{n_k\}_{k\geq 1}$ of natural numbers such that $\mu_{n_k}\Rightarrow \mu$ for some $\mu\in\mathscr{P}(V_S)$ and $E_{n_k}\equiv E_0$ for some $E_0\in \mathscr{S}_0$, where $\text{supp}(\mu|_E)\subseteq E_0.$ Note for all $z\in\partial (\{\overrightarrow{uv}: uv\in E_{0}\})$, $\lim_{j\uparrow\infty}f_{n_j}(z)=\infty.$ Then, for any $M>0$, when $k$ is large enough, $f_{n_k}(z)\geq M,\ z\in\partial (\{\overrightarrow{uv}: uv\in E_{0}\});$ and further
\begin{align*}
\lim_{\delta\to 0}\alpha_c(\delta)&=\lim_{k\to\infty}\alpha_c(1/{n_k})=\lim_{k\to\infty}\left\{\Lambda_1(\mu_{n_k})+\int_{\partial (\{\overrightarrow{uv}: uv\in E_{0}\})}f_{n_k}\ {\rm d}\mu_{n_k}\right\}\\
&\ge\liminf_{k\to\infty}\left\{\Lambda_1(\mu_{n_k})+M\mu_{n_k}(\partial (\{\overrightarrow{uv}: uv\in E_{0}\}))\right\}\\
&\ge\Lambda_1(\mu)+M\mu(\partial (\{\overrightarrow{uv}: uv\in E_{0}\})).
\end{align*}
This forces $\mu(\{\overrightarrow{uv}: uv\in \partial E_{0}\})=0$. Otherwise $\mu(\partial (\{\overrightarrow{uv}: uv\in E_{0}\}))>0$ by \eqref{prop-partial-2}, which implies $b\geq \lim\limits_{\delta\to 0}\alpha_c(\delta)=\infty.$ Therefore, $\text{supp}(\mu|_E)\subseteq E_0,\text{supp}(\mu|_E)\cap \partial E_0=\emptyset,$ and $\lim\limits_{\delta\to 0}\alpha_c(\delta)\ge\Lambda_1(\mu)\geq b.$ This completes the proof.\end{proof}

\end{appendix}

	{}

\end{document}